\documentclass{amsart}
\usepackage{amsmath}
\usepackage{amssymb}
\usepackage{amsfonts}

\newtheorem{theorem}{Theorem}
\theoremstyle{plain}

\newtheorem{corollary}{Corollary}

\newtheorem{definition}{Definition}

\newtheorem{lemma}{Lemma}
\newtheorem{notation}{Notation}

\newtheorem{remark}{Remark}

\numberwithin{equation}{section}
\input{tcilatex}

\begin{document}
\title[]{A characterization of rough fractional type integral operators and
Campanato estimates for their commutators on the variable exponent vanishing
generalized Morrey spaces }
\author{FER\.{I}T G\"{U}RB\"{U}Z, SHENGHU D\.{I}NG, HU\.{I}L\.{I} HAN, P\.{I}%
NHONG LONG}
\address{HAKKARI UNIVERSITY, FACULTY OF EDUCATION, DEPARTMENT OF MATHEMATICS
EDUCATION, HAKKARI 30000, TURKEY }
\email{feritgurbuz@hakkari.edu.tr}
\address{NINGXIA UNIVERSITY, SCHOOL OF MATHEMATICS AND STATISTICS, YINCHUAN
750021, P.R. CHINA}
\email{dshnx2006@163.com, nxhan126@126.com, longph@nxu.edu.cn}
\urladdr{}
\thanks{}
\curraddr{ }
\urladdr{}
\thanks{}
\date{}
\dedicatory{}
\thanks{}

\begin{abstract}
In this paper, applying some properties of variable exponent analysis, we
first dwell on Adams and Spanne type estimates for a class of fractional
type integral operators of variable orders, respectively and then, obtain
variable exponent generalized Campanato estimates for the corresponding
commutators on the vanishing generalized Morrey spaces $VL_{\Pi }^{p\left(
\cdot \right) ,w\left( \cdot \right) }\left( E\right) $ with variable
exponent $p(\cdot )$ and bounded set $E$. In fact, the results in this paper
are generalizations of some known results on an operator basis.
\end{abstract}

\maketitle

\section{Introduction}

In this paper we mainly focus on some operators and commutators on the
variable exponent generalized Morrey type space. Precisely, our aim is to
characterize the boundedness for the maximal operator, fractional integral
operator and fractional maximal operator with rough kernel as well as the
corresponding commutators on the variable exponent vanishing generalized
Morrey spaces.

Now, we list some background material needed for later sections. We assume
that our readers are familiar with the foundation of real analysis. Since it
is impossible to squeeze everything into just a few pages, sometimes we will
refer the interested readers to some papers and references.

\begin{notation}
$\cdot $ Let $x=\left( x_{1},x_{2},\ldots ,x_{n}\right) $, $\xi =\left( \xi
_{1},\xi _{2},\ldots ,\xi _{n}\right) $ . . . . etc. be points of the real $%
n $-dimensional space ${\mathbb{R}^{n}}$. Let $x.\xi
=\dsum\limits_{i=1}^{n}x_{i}\xi _{i\text{ }}$stand for the usual dot product
in ${\mathbb{R}^{n}}$ and $\left\vert x\right\vert =\left(
\dsum\limits_{i=1}^{n}x_{i}^{2}\right) ^{\frac{1}{2}}$ for the Euclidean
norm of $x$.

$\cdot $ By $x^{\prime }$, we always mean the unit vector corresponding to $%
x $, i.e. $x^{\prime }=\frac{x}{|x|}$ for any $x\neq 0$.

$\cdot $ $S^{n-1}=\left\{ x\in {\mathbb{R}^{n}:}|x|=1\right\} $ represents
the unit sphere in Euclidean $n$-dimensional space ${\mathbb{R}^{n}}$ $%
(n\geq 2)$ and $dx^{\prime }$ is its surface measure.

$\cdot $ Denote by $\left\vert E\right\vert $ the Lebesgue measure and by $%
\chi _{E}$ the characteristic function for a measurable set $E\subset {%
\mathbb{R}^{n}}$.

$\cdot $ Given a function $f$, we denote the mean value of $f$ on $E$ by%
\begin{equation*}
f_{E}:=\frac{1}{\left\vert E\right\vert }\dint\limits_{E}f\left( x\right) dx.
\end{equation*}%
$\cdot $ $B(x,r)=\left\{ y\in {\mathbb{R}^{n}:}\left\vert x-y\right\vert
<r\right\} $ denotes $x$-centred Euclidean ball with radius $r$, $B^{C}(x,r)$
denotes its complement and $|B(x,r)|$ is the Lebesgue measure of the ball $%
B(x,r)$, $|B(x,r)|=v_{n}r^{n}$, where $v_{n}=|B(0,1)|=\frac{2\pi ^{\frac{n}{2%
}}}{n\Gamma \left( \frac{n}{2}\right) }$ and $\tilde{B}(x,r)=B(x,r)\cap E$,
where $E\subset {\mathbb{R}^{n}}$ is an open set. Finally, we use the
notation 
\begin{equation*}
f_{B\left( x,r\right) }=\frac{1}{\left\vert B\left( x,r\right) \right\vert }%
\int\limits_{\tilde{B}(x,r)}f\left( y\right) dy.
\end{equation*}

$\cdot $ $C$ stands for a positive constant that can change its value in
each statement without explicit mention.

$\cdot $ The exponents $p^{\prime }\left( \cdot \right) $ and $s^{\prime
}\left( \cdot \right) $ always denote the conjugate index of any exponent $%
1<p\left( x\right) <\infty $ and $1<s\left( x\right) <\infty $, that is, $%
\frac{1}{p^{\prime }\left( x\right) }:=1-\frac{1}{p\left( x\right) }$ and $%
\frac{1}{s^{\prime }\left( x\right) }:=1-\frac{1}{s\left( x\right) }$.

$\cdot $ In the sequel, for any exponent $1<p\left( x\right) <\infty $ and
bounded sets $E\subset $ ${\mathbb{R}^{n}}$, if we use 
\begin{equation}
\left\vert p\left( x\right) -p\left( y\right) \right\vert \leq \frac{-C}{%
\log \left( \left\vert x-y\right\vert \right) }\qquad \left\vert
x-y\right\vert \leq \frac{1}{2},\qquad x,y\in E{,}  \label{1}
\end{equation}%
where $C=C\left( p\right) >0$ does not depend on $x$, $y$, then we call that 
$p\left( \cdot \right) $ satisfies local log-H\"{o}lder continuity condition
or Dini-Lipschitz condition. The important role of local log-H\"{o}lder
continuity of $p\left( x\right) $ is well known in variable analysis. On the
other hand, the condition%
\begin{equation*}
\left\vert p\left( x\right) -p\left( y\right) \right\vert \leq \frac{C}{\log
\left( e+\left\vert x\right\vert \right) }\qquad \left\vert y\right\vert
\geq \left\vert x\right\vert ,\qquad x,y\in E,
\end{equation*}%
introduced by Cruz-Uribe et al. in \cite{U} is known as the log-H\"{o}lder
decay condition used for unbounded sets $E$. It is equivalent to the
condition that there exists a number $p_{\infty }\in \left[ 1,\infty \right) 
$ such that%
\begin{equation}
\left\vert \frac{1}{p_{\infty }}-\frac{1}{p\left( x\right) }\right\vert \leq 
\frac{C_{\infty }}{\log \left( e+\left\vert x\right\vert \right) }\qquad 
\text{for all }x\in E{.}  \label{2}
\end{equation}%
where $p_{\infty }=\lim\limits_{\left\vert x\right\vert \rightarrow \infty
}p\left( x\right) $.

If $p\left( \cdot \right) $ satisfies both (\ref{1}) and (\ref{2}), then we
say that it is log-H\"{o}lder continuous.

$\cdot $ Here and henceforth, $F\approx G$ means $F\gtrsim G\gtrsim F$;
while $F\gtrsim G$ means $F\geq CG$ for a constant $C>0$.

$\cdot $ Let $\Omega \in L_{s}(S^{n-1})$ with $1<s\leq \infty $ be
homogeneous function of degree $0$ on ${\mathbb{R}^{n}}$ and satisfy the
integral zero property over the unit sphere $S^{n-1}$. Moreover, note that $%
\left\Vert \Omega \right\Vert _{L_{s}\left( S^{n-1}\right) }:=\left(
\int\limits_{S^{n-1}}\left\vert \Omega \left( z^{\prime }\right) \right\vert
^{s}d\sigma \left( z^{\prime }\right) \right) ^{\frac{1}{s}}$ and%
\begin{eqnarray}
\left\Vert \Omega \left( z-y\right) \right\Vert _{L_{s}\left( \tilde{B}%
\left( x,r\right) \right) } &=&\left( \dint\limits_{\tilde{B}\left(
x,r\right) }\Omega \left( \left( z-y\right) \right) ^{s}dz\right) ^{\frac{1}{%
s}}  \notag \\
&\lesssim &\left( \dint\limits_{\tilde{B}\left( x,r\right) }\Omega \left(
\sigma \right) ^{s}\dint\limits_{0}^{r}\rho ^{n-1}d\rho d\sigma \right) ^{%
\frac{1}{s}}  \notag \\
&\lesssim &\left\Vert \Omega \right\Vert _{L_{s}\left( S^{n-1}\right) }r^{%
\frac{n}{s}},  \label{0*}
\end{eqnarray}%
for $z\in B(x,r)$.

$\cdot $ Suppose that $0<\alpha \left( x\right) <n$, $x\in E\subset {\mathbb{%
R}^{n}}$. Then, the rough Riesz type potential operator with variable order $%
I_{\Omega ,\alpha \left( \cdot \right) }$ and the corresponding rough
fractional maximal operator with variable order $M_{\Omega ,\alpha \left(
\cdot \right) }$ are defined, respectively, by%
\begin{equation*}
I_{\Omega ,\alpha \left( \cdot \right) }f(x)=\int\limits_{E}\frac{\Omega
(x-y)}{|x-y|^{n-\alpha \left( x\right) }}f(y)dy
\end{equation*}%
and%
\begin{equation*}
M_{\Omega ,\alpha \left( \cdot \right) }f(x)=\sup_{r>0}\left\vert
B(x,r)\right\vert ^{\frac{\alpha \left( x\right) }{n}-1}\int\limits_{\tilde{B%
}(x,r)}\left\vert \Omega \left( x-y\right) \right\vert |f(y)|dy,
\end{equation*}%
where $E\subset {\mathbb{R}^{n}}$ is an open set. On the other hand, if $%
\alpha \left( \cdot \right) =0$, then the rough Calder\'{o}n-Zygmund type
singular integral operator $T_{\Omega }$ in the sense of principal value
Cauchy integral is defined by 
\begin{equation*}
T_{\Omega }f(x)=p.v.\int\limits_{E}\frac{\Omega (x-y)}{|x-y|^{n}}f(y)dy,
\end{equation*}%
and especially in the limiting case $\alpha \left( \cdot \right) =0$, the
rough fractional maximal operator with variable order $M_{\Omega ,\alpha }$
reduces to the rough Hardy-Littlewood maximal operator $M_{\Omega }$ and $%
M_{\Omega }$ is also defined by%
\begin{equation*}
M_{\Omega }f\left( x\right) =\sup_{r>0}\frac{1}{\left\vert B(x,r)\right\vert 
}\int\limits_{\tilde{B}(x,r)}\left\vert \Omega \left( y\right) \right\vert
\left\vert f\left( x-y\right) \right\vert dy,
\end{equation*}%
where $E\subset {\mathbb{R}^{n}}$ is an open set. In fact, we can easily see
that when $\Omega \equiv 1$; $M_{1,\alpha \left( \cdot \right) }\equiv
M_{\alpha \left( \cdot \right) }$ and $I_{1,\alpha \left( \cdot \right)
}\equiv I_{\alpha \left( \cdot \right) }$ are the fractional maximal
operator with variable order and the Riesz type potential operator with
variable order, and similarly $M$ and $T$ are the Hardy-Littlewood maximal
operator and the standard Calder\'{o}n-Zygmund type singular integral
operator, respectively.

$\cdot $ Let $b$ be a locally integral function on $E$. Define the rough
commutators $\left[ b,T_{\Omega }\right] $, $\left[ b,M_{\Omega }\right] $
generated by the function $b$ and the operators $T_{\Omega }$, $M_{\Omega }$
with rough kernel $\Omega $ via%
\begin{eqnarray*}
\left[ b,T_{\Omega }\right] f\left( x\right) &=&b\left( x\right) T_{\Omega
}f\left( x\right) -T_{\Omega }\left( bf\right) \left( x\right) \\
&=&p.v.\int\limits_{E}\frac{\Omega (x-y)}{|x-y|^{n}}\left( b\left( x\right)
-b\left( y\right) \right) f(y)dy
\end{eqnarray*}%
and 
\begin{eqnarray*}
\left[ b,M_{\Omega }\right] f\left( x\right) &=&b\left( x\right) M_{\Omega
}f\left( x\right) -M_{\Omega }\left( bf\right) \left( x\right) \\
&=&\sup_{r>0}\frac{1}{\left\vert B(x,r)\right\vert }\int\limits_{\tilde{B}%
(x,r)}\left\vert \Omega \left( x-y\right) \right\vert \left\vert b\left(
x\right) -\left\vert b\left( y\right) \right\vert \right\vert \left\vert
f(y)\right\vert dy,
\end{eqnarray*}%
similarly, define the rough commutators $\left[ b,I_{\Omega ,\alpha \left(
\cdot \right) }\right] $, $\left[ b,M_{\Omega ,\alpha \left( \cdot \right) }%
\right] $ generated by the function $b$ and the fractional integral operator 
$I_{\Omega ,\alpha \left( \cdot \right) }$, the fractional maximal operator $%
M_{\Omega ,\alpha \left( \cdot \right) }$ with rough kernel $\Omega $ and
variable order $\alpha (\cdot )$($0\leq \alpha \left( \cdot \right) <n$) as
follows.

\begin{eqnarray*}
\left[ b,I_{\Omega ,\alpha \left( \cdot \right) }\right] f\left( x\right)
&=&b\left( x\right) I_{\Omega ,\alpha \left( \cdot \right) }f\left( x\right)
-I_{\Omega ,\alpha \left( \cdot \right) }\left( bf\right) \left( x\right) \\
&=&\int\limits_{E}\frac{\Omega (x-y)}{|x-y|^{n-\alpha \left( x\right) }}%
\left( b\left( x\right) -b\left( y\right) \right) f(y)dy
\end{eqnarray*}%
and 
\begin{eqnarray*}
\left[ b,M_{\Omega ,\alpha \left( \cdot \right) }\right] f\left( x\right)
&=&b\left( x\right) M_{\Omega ,\alpha \left( \cdot \right) }f\left( x\right)
-M_{\Omega ,\alpha \left( \cdot \right) }\left( bf\right) \left( x\right) \\
&=&\sup_{r>0}\left\vert B(x,r)\right\vert ^{\frac{\alpha \left( x\right) }{n}%
-1}\int\limits_{\tilde{B}(x,r)}\left\vert \Omega \left( x-y\right)
\right\vert \left\vert b\left( x\right) -\left\vert b\left( y\right)
\right\vert \right\vert \left\vert f(y)\right\vert dy.
\end{eqnarray*}
\end{notation}

Morrey spaces can complement the boundedness properties of operators that
Lebesgue spaces can not handle. Morrey spaces which we have been handling
are called classical Morrey spaces(see \cite{Morrey}). In this sense, the
classical Morrey spaces(see \cite{Morrey}) ever were applied to study the
local regularity behavior of solutions to second order elliptic partial
differential equations (see \cite{Gi} and \cite{Ta}). For the boundedness of
various classical operators in Morrey or Morrey type spaces, refer to for
maximal, potential, singular integral and others, \cite{Ad, AX, AX1, AX2,
AM, CF, KNS, So} and references therein. In \cite{Vi} the vanishing Morrey
space was introduced by Vitanza to character the regularity results for
elliptic partial differential equations. Moreover, Ragusa(\cite{Ra1})and
Samko et al(\cite{Sa2, Sa3} and references therein) ever systematically
obtain the boundedness of various classical operators in such these spaces.
Recently, while we try out to resolve somewhat modern problems emerging
inherently such that nonlinear elasticity theory, fluid mechanics etc., it
has become that classical function spaces are not anymore suitable spaces.
It thus became essential to introduce and analysis the diverse function
spaces from diverse viewpoints. One of such spaces is the variable exponent
Lebesgue space $L^{p\left( \cdot \right) }$. This space is a generalization
of the classical $L^{p}\left( 
\mathbb{R}
^{n}\right) $ space, in which the constant exponent $p$ is replaced by an
exponent function $p\left( \cdot \right) :%
\mathbb{R}
^{n}\rightarrow \left( 0,\infty \right) $, it consists of all functions $f$
such that $\dint\limits_{%
\mathbb{R}
^{n}}\left\vert f\left( x\right) \right\vert ^{p\left( x\right) }dx$. This
theory got a boost in 1931 when Orlicz published his seminal paper \cite%
{Orlicz}. The next major step in the investigation of variable exponent
spaces was the comprehensive paper by Kov\'{a}\v{c}ik and R\'{a}kosn\'{\i}k
in the early 90's \cite{Kovacik}. Since then, the theory of variable
exponent spaces was applied to many fields, refer to \cite{CLR, Wu} for the
image processing, \cite{AR} for thermorheological fluids, \cite{Ru} for
electrorheological fluids and \cite{HHLN} for the differential equations
with nonstandard growth. For the nonweighted and weighted variable exponent
settings, refer to \cite{DHHR,DHN,DR,FZh}. On the other hand, Kov\'{a}\v{c}%
ik and R\'{a}kosn\'{\i}k \cite{Kovacik} established many of the basic
properties of Lebesgue and Sobolev spaces. Moreover, since these authors
clarified fundamental properties of the variable exponent Lebesgue and
Sobolev spaces, there are many spaces studied, such as variable exponent
Morrey, generalized Morrey, vanishing generalized Morrey, Herz-Morrey
spaces, etc. see \cite{Almeida, Guliyev, Ho, Long, Wu1}. In the last decade,
when the parameters that define the operator have changed from point to
point, there has been a strong interest in fractional type operators and the
"variable setting" function spaces. The field called variable exponent
analysis has become a fairly branched area with many interesting results
obtained in the last decade such as harmonic analysis, approximation theory,
operator theory, pseudo-differential operators, etc. But, the results in
this paper lie in these spaces known as variable exponent Morrey type spaces
on the rough fractional type operators with variable order of harmonic
analysis, which has been extensively developed during the last ten years and
continues to attract attention of researchers from various fields of
mathematics. Many of problems about such spaces have been solved both in the
classical setting and in the Euclidean setting, including fractional upper
and lower dimensions. For example, in 2008 variable exponent Morrey spaces $%
L^{p\left( \cdot \right) ,\lambda \left( \cdot \right) }$ were introduced to
study the boundedness of $M$ and $I_{\alpha \left( \cdot \right) }$ in the
Euclidean setting by Almeida et al. \cite{Almeida}. In 2010, variable
exponent generalized Morrey spaces $L^{p\left( \cdot \right) ,w\left( \cdot
\right) }\left( E\right) $ were introduced to consider the boundedness of $M$%
, $I_{\alpha \left( \cdot \right) }$, $T$ for bounded sets $E\subset $ ${%
\mathbb{R}^{n}}$ on $\ L^{p\left( \cdot \right) ,w\left( \cdot \right)
}\left( E\right) $ in \cite{Guliyev}. In 2016, variable exponent vanishing
generalized Morrey spaces $VL_{\Pi }^{p\left( \cdot \right) ,w\left( \cdot
\right) }\left( E\right) $ were introduced to characterize the boundedness
of $M$, $I_{\alpha \left( \cdot \right) }$, $T$ for bounded or unbounded
sets $E$ on $VL_{\Pi }^{p\left( \cdot \right) ,w\left( \cdot \right) }\left(
E\right) $ in \cite{Long}.

After the boundedness of $M$, $I_{\alpha \left( \cdot \right) }$, $T$ for
bounded sets $E\subset $ ${\mathbb{R}^{n}}$ both on $L^{p\left( \cdot
\right) ,w\left( \cdot \right) }\left( E\right) $ and $VL_{\Pi }^{p\left(
\cdot \right) ,w\left( \cdot \right) }\left( E\right) $ have been
established in \cite{Guliyev, Long}, a natural question is: Can these
results be generalized? In other words, what properties do the more general
operators $M_{\Omega ,\alpha \left( \cdot \right) }$ and $I_{\Omega ,\alpha
\left( \cdot \right) }$ have for bounded sets $E\subset $ ${\mathbb{R}^{n}}$
both on $L^{p\left( \cdot \right) ,w}\left( E\right) $ and $VL_{\Pi
}^{p\left( \cdot \right) ,w}\left( E\right) $? We give answers to these
questions in this paper. In view of the definitions of $M_{\Omega }$, $%
I_{\Omega ,\alpha \left( \cdot \right) }$ and $T_{\Omega }$ above, we see
that these operators are generalizations of the operators $M$, $I_{\alpha
\left( \cdot \right) }$, $T$. On the other hand, recently, Rafeiro and Samko 
\cite{Rafeiro} proved that the boundedness of $I_{\Omega ,\alpha \left(
\cdot \right) }$, $M_{\Omega ,\alpha \left( \cdot \right) }$ and $M_{\Omega
} $ for bounded sets $E\subset $ ${\mathbb{R}^{n}}$ both on $L^{p\left(
\cdot \right) }$ and $L^{p\left( \cdot \right) ,\lambda \left( \cdot \right)
}$, respectively.

\section{Preliminaries and Main results}

In this section, we recall the definitions and some properties of basic
spaces that we need and also give the main results.

\subsection{Preliminaries on \textbf{variable exponent Lebesgue spaces }$%
L^{p\left( \cdot \right) }$}

\begin{flushleft}
We first define variable exponent Lebesgue space.
\end{flushleft}

\begin{definition}
Given an open set $E\subset {\mathbb{R}^{n}}$ and a measurable function $%
p\left( \cdot \right) :E\rightarrow \left[ 1,\infty \right) $. We assume
that $1\leq p_{-}\left( E\right) \leq p_{+}\left( E\right) <\infty $, where $%
p_{-}\left( E\right) =\limfunc{essinf}\limits_{x\in E}p\left( x\right) $ and 
$p_{+}\left( E\right) =\limfunc{esssup}\limits_{x\in E}p\left( x\right) $.
The variable exponent Lebesgue space $L^{p\left( \cdot \right) }\left(
E\right) $ is the collection of all measurable functions $f$ such that, for
some $\lambda >0$, $\rho \left( f/\lambda \right) <\infty $, where the
modular is defined by%
\begin{equation*}
\rho \left( f\right) =\rho _{p\left( \cdot \right) }\left( f\right) =\dint
\limits_{E}\left \vert f\left( x\right) \right \vert ^{p\left( x\right) }dx.
\end{equation*}%
Then, the spaces $L^{p\left( \cdot \right) }\left( E\right) $ and $%
L_{loc}^{p\left( \cdot \right) }\left( E\right) $ are defined by%
\begin{equation*}
L^{p\left( \cdot \right) }\left( E\right) =\left \{ f\text{ is measurable}%
:\rho _{p\left( \cdot \right) }\left( f/\lambda \right) <\infty \text{ for
some }\lambda >0\right \}
\end{equation*}%
and%
\begin{equation*}
L_{loc}^{p\left( \cdot \right) }\left( E\right) =\left \{ f\text{ is
measurable}:f\in L^{p\left( \cdot \right) }\left( K\right) \text{ for all
compact }K\subset E\right \} ,
\end{equation*}%
with the Luxemburg norm 
\begin{equation}
\left \Vert f\right \Vert _{L^{p\left( \cdot \right) }\left( E\right) }=\inf
\left \{ \lambda >0:\rho _{p\left( \cdot \right) }\left( f/\lambda \right)
=\dint \limits_{E}\left( \left \vert f\left( x\right) \right \vert /\lambda
\right) ^{p\left( x\right) }dx\leq 1\right \} \qquad f\in L^{p\left( \cdot
\right) }\left( E\right) .  \label{0}
\end{equation}%
Since $p_{-}\left( E\right) \geq 1$, $\left \Vert \cdot \right \Vert
_{L^{p\left( \cdot \right) }\left( E\right) }$ is a norm and $\left(
L^{p\left( \cdot \right) }\left( E\right) ,\left \Vert \cdot \right \Vert
_{L^{p\left( \cdot \right) }\left( E\right) }\right) $ is a Banach space.
However, if $p_{-}\left( E\right) <1$, then $\left \Vert \cdot \right \Vert
_{L^{p\left( \cdot \right) }\left( E\right) }$ is a quasinorm and $\left(
L^{p\left( \cdot \right) }\left( E\right) ,\left \Vert \cdot \right \Vert
_{L^{p\left( \cdot \right) }\left( E\right) }\right) $ is a quasi Banach
space. The variable exponent norm has the following property%
\begin{equation*}
\left \Vert f^{\lambda }\right \Vert _{L^{p\left( \cdot \right) }\left(
E\right) }=\left \Vert f\right \Vert _{L^{\lambda p\left( \cdot \right)
}\left( E\right) }^{\lambda },
\end{equation*}%
for $\lambda \geq \frac{1}{p_{-}}$. Moreover, these spaces are referred to
as variable $L^{p}$ spaces, since they generalize the standard $L^{p}$
spaces: if $p\left( x\right) =p$ is constant, then $L^{p\left( \cdot \right)
}\left( E\right) $ is isometrically isomorphic to $L^{p}\left( E\right) $.
As a result, using notations above ($p_{-}\left( E\right) $ and $p_{+}\left(
E\right) $), we define a class of variable exponent as follows:%
\begin{equation*}
\Phi \left( E\right) =\left \{ p\left( \cdot \right) :E\rightarrow \left[
1,\infty \right) ,\text{ }p_{-}\left( E\right) \geq 1\text{, }p_{+}\left(
E\right) <\infty \right \} .
\end{equation*}
\end{definition}

Now, we define two the sets of exponents $p\left( x\right) $ with $1\leq
p_{-}\left( E\right) \leq p_{+}\left( E\right) <\infty $. These will be
denoted by as follows:%
\begin{equation*}
\mathcal{P}^{\log }\left( E\right) =\left\{ 
\begin{array}{c}
p\left( \cdot \right) :p_{-}\left( E\right) \geq 1\text{, }p_{+}\left(
E\right) <\infty \text{ and }p\left( \cdot \right) \text{ satisfy both the
conditions (\ref{1}) and (\ref{2}) } \\ 
\text{(the latter required if }E\text{ is unbounded)}%
\end{array}%
\right\}
\end{equation*}%
and%
\begin{equation*}
\mathcal{B}\left( E\right) =\left\{ p\left( \cdot \right) :p\left( \cdot
\right) \in \mathcal{P}^{\log }\left( E\right) \text{, }M\text{ is bounded
on }L^{p\left( \cdot \right) }\left( E\right) \right\} ,
\end{equation*}%
where $M$ is the Hardy-Littlewood maximal operator. We recall that the
generalized H\"{o}lder inequality on Lebesgue spaces with variable exponent 
\begin{equation*}
\left\vert \dint\limits_{E}f\left( x\right) g\left( x\right) dx\right\vert
\leq \dint\limits_{E}\left\vert f\left( x\right) g\left( x\right)
\right\vert dx\leq C_{p}\left\Vert f\right\Vert _{L^{p\left( \cdot \right)
}\left( E\right) }\left\Vert g\right\Vert _{L^{p^{\prime }\left( \cdot
\right) }\left( E\right) }\qquad C_{p}=1+\frac{1}{p_{-}}-\frac{1}{p_{+}},
\end{equation*}%
is known to hold for $p\left( \cdot \right) \in \Phi \left( E\right) $, $%
f\in L^{p\left( \cdot \right) }\left( E\right) $ and $g\in L^{p^{\prime
}\left( \cdot \right) }\left( E\right) $, see Theorem 2.1 in \cite{Kovacik}.
Now, we recall some recent results for the rough Riesz type potential
operator with variable order $I_{\Omega ,\alpha \left( \cdot \right) }$ and
the corresponding rough fractional maximal operator with variable order $%
M_{\Omega ,\alpha \left( \cdot \right) }$ on variable exponent Lebesgue
space $L^{p\left( \cdot \right) }\left( E\right) $. The order $\alpha \left(
x\right) $ of the potential is not assumed to be continuous. We assume that
it is a measurable function on $E$ satisfying the following assumptions%
\begin{equation}
\left. 
\begin{array}{c}
\alpha _{0}=\limfunc{essinf}\limits_{x\in E}\alpha \left( x\right) >0 \\ 
\limfunc{esssup}\limits_{x\in E}\alpha \left( x\right) p\left( x\right) <n%
\end{array}%
\right\} .  \label{25}
\end{equation}%
First, the norm in the space $L^{p\left( \cdot \right) }\left( E\right) $
seems to be complicated in a sense, to be calculated or estimated. So the
following basic estimation of the boundedness of an operator $B$:%
\begin{equation}
\left\Vert Bf\right\Vert _{L^{p\left( \cdot \right) }\left( E\right)
}\lesssim \left\Vert f\right\Vert _{L^{p\left( \cdot \right) }\left(
E\right) }  \label{62}
\end{equation}%
is not easy. However, in the case of linear operators, the above inequality
between the norm and the modular and the homogeneity property%
\begin{equation*}
\left\Vert B\right\Vert _{X\rightarrow X}=\sup_{f\in X}\frac{\left\Vert
Bf\right\Vert _{X}}{\left\Vert f\right\Vert _{X}}=\sup_{\left\Vert
f\right\Vert _{X}=1}\left\Vert Bf\right\Vert _{X}
\end{equation*}%
allow us to replace checking of (\ref{62}) by a work with a modular:%
\begin{equation*}
\dint\limits_{E}\left\vert Bf\left( x\right) \right\vert ^{p\left( x\right)
}dx\text{, for all }f\text{ with }\left\Vert f\right\Vert _{L^{p\left( \cdot
\right) }\left( E\right) }\leq 1,
\end{equation*}%
which is certainly easier. In that respect, the boundedness of the rough
Riesz-type potential operator from the space $L^{p\left( \cdot \right)
}\left( 
\mathbb{R}
^{n}\right) $ with the variable exponent $p(x)$ into the space $L^{q\left(
\cdot \right) }\left( 
\mathbb{R}
^{n}\right) $ with the limiting Sobolev exponent%
\begin{equation}
\frac{1}{q\left( x\right) }=\frac{1}{p\left( x\right) }-\frac{\alpha \left(
x\right) }{n}  \label{26}
\end{equation}%
was an open problem for a long time. It was solved in the case of bounded
domains. First, in \cite{Rafeiro}, in the case of bounded domains $E$, there
has the following conditional result.

\begin{theorem}
\label{teo1}Let $E$ be a bounded open set, $\Omega \in L_{s}(S^{n-1})$ with $%
1<s\leq \infty $, $p\left( x\right) \in \mathcal{P}^{\log }\left( E\right) $%
, $\alpha \left( x\right) $ satisfy the assumptions (\ref{25}) and $\left(
p^{\prime }\right) _{+}\leq s$. Define $q\left( x\right) $ by (\ref{26}).
Then, the rough Riesz-type potential operator $I_{\Omega ,\alpha \left(
\cdot \right) }$ is $\left( L^{p\left( \cdot \right) }\left( E\right)
\rightarrow L^{q\left( \cdot \right) }\left( E\right) \right) $-bounded,
that is, the Sobolev type theorem%
\begin{equation}
\left\Vert I_{\Omega ,\alpha \left( \cdot \right) }f\right\Vert _{L^{q\left(
\cdot \right) }\left( E\right) }\lesssim \left\Vert f\right\Vert
_{L^{p\left( \cdot \right) }\left( E\right) }  \label{5}
\end{equation}%
is valid.
\end{theorem}

\begin{corollary}
\label{Corollary0}Let $E$ be a bounded open set, $\Omega \in L_{s}(S^{n-1})$
with $1<s\leq \infty $ be homogeneous function of degree $0$ on ${\mathbb{R}%
^{n}}$, $\frac{p}{s^{\prime }}\in \mathcal{B}\left( E\right) $ and $\left(
p^{\prime }\right) _{+}\leq s$. Under the conditions of Theorem \ref{teo1}
(taking $\alpha \left( \cdot \right) =0$ there), the operator $T_{\Omega }$
is $\left( L^{p\left( \cdot \right) }\left( E\right) \rightarrow L^{p\left(
\cdot \right) }\left( E\right) \right) $-bounded, that is,%
\begin{equation}
\left\Vert T_{\Omega }f\right\Vert _{L^{p\left( \cdot \right) }\left(
E\right) }\lesssim \left\Vert f\right\Vert _{L^{p\left( \cdot \right)
}\left( E\right) }  \label{5*}
\end{equation}%
is valid.
\end{corollary}

On the other hand, the pointwise inequalities on variable exponent Lebesgue
spaces are very useful. Indeed, we have%
\begin{equation*}
\left\vert f\left( x\right) \right\vert \leq \left\vert h\left( x\right)
\right\vert \text{ implies that }\left\Vert f\right\Vert _{L^{p\left( \cdot
\right) }\left( E\right) }\lesssim \left\Vert h\right\Vert _{L^{p\left(
\cdot \right) }\left( E\right) }.
\end{equation*}%
Thus, if one operator is pointwise dominated by another one:%
\begin{equation*}
\left\vert Bf\left( x\right) \right\vert \leq \left\vert Df\left( x\right)
\right\vert ,
\end{equation*}%
and we know that the operator $D$ is bounded, then the boundedness of the
operator $B$ immediately follows. For example, by Theorem \ref{teo1} we get
the following:

\begin{theorem}
\label{teo3}Under the conditions of Theorem \ref{teo1}, the operator $%
M_{\Omega ,\alpha \left( \cdot \right) }$ is $\left( L^{p\left( \cdot
\right) }\left( E\right) \rightarrow L^{q\left( \cdot \right) }\left(
E\right) \right) $-bounded, that is, the Sobolev type theorem%
\begin{equation}
\left\Vert M_{\Omega ,\alpha \left( \cdot \right) }f\right\Vert _{L^{q\left(
\cdot \right) }\left( E\right) }\lesssim \left\Vert f\right\Vert
_{L^{p\left( \cdot \right) }\left( E\right) }  \label{6}
\end{equation}%
is valid.
\end{theorem}

\begin{corollary}
Let $E$ be a bounded open set, $\Omega \in L_{s}(S^{n-1})$ with $1<s\leq
\infty $ be homogeneous function of degree $0$ on ${\mathbb{R}^{n}}$, $\frac{%
p}{s^{\prime }}\in \mathcal{B}\left( E\right) $ and $\left( p^{\prime
}\right) _{+}\leq s$. Under the conditions of Theorem \ref{teo3} (taking $%
\alpha \left( \cdot \right) =0$ there), the operator $M_{\Omega }$ is $%
\left( L^{p\left( \cdot \right) }\left( E\right) \rightarrow L^{p\left(
\cdot \right) }\left( E\right) \right) $-bounded, that is,%
\begin{equation}
\left\Vert M_{\Omega }f\right\Vert _{L^{p\left( \cdot \right) }\left(
E\right) }\lesssim \left\Vert f\right\Vert _{L^{p\left( \cdot \right)
}\left( E\right) }  \label{6*}
\end{equation}%
is valid.
\end{corollary}

We are now in a place of proving (\ref{6}) in Theorem \ref{teo3}.

\begin{remark}
The conclusion of (\ref{6}) is a direct consequence of the following Lemma %
\ref{lemma100} and (\ref{5}). In order to do this, we need to define an
operator by%
\begin{equation*}
\widetilde{T}_{\left \vert \Omega \right \vert ,\alpha \left( \cdot \right)
}\left( \left \vert f\right \vert \right) (x)=\int \limits_{E}\frac{\left
\vert \Omega (x-y)\right \vert }{|x-y|^{n-\alpha \left( x\right) }}\left
\vert f(y)\right \vert dy\qquad 0<\alpha \left( x\right) <n,
\end{equation*}%
where $\Omega \in L_{s}(S^{n-1})\left( s>1\right) $ is homogeneous of degree
zero on ${\mathbb{R}^{n}}$.
\end{remark}

Using the idea of proving Corollary 3.1. in \cite{Wu1}, we can obtain the
following pointwise relation:

\begin{lemma}
\label{lemma100}Let $0<\alpha \left( x\right) <n$ and $\Omega \in
L_{s}(S^{n-1})\left( s>1\right) $. Then we have%
\begin{equation}
M_{\Omega ,\alpha \left( \cdot \right) }\left( f\right) (x)\leq C\widetilde{T%
}_{\left \vert \Omega \right \vert ,\alpha \left( \cdot \right) }\left(
\left \vert f\right \vert \right) (x)\qquad \text{for }x\in {\mathbb{R}^{n}},
\label{39}
\end{equation}%
where $C$ does not depend on $f$ and $x$.
\end{lemma}

\begin{proof}
To prove (\ref{39}), we observe that for any $x\in {\mathbb{R}^{n}}$, there
exists an $r=r_{x}$ such that%
\begin{equation*}
M_{\Omega ,\alpha \left( \cdot \right) }\left( f\right) (x)\leq \frac{2}{%
|B\left( x,r_{x}\right) |^{n-\alpha \left( x\right) }}\int \limits_{B\left(
x,r_{x}\right) }\left \vert \Omega \left( x-y\right) \right \vert |f(y)|dy,
\end{equation*}%
and by the inequality above, we get 
\begin{eqnarray*}
\widetilde{T}_{\left \vert \Omega \right \vert ,\alpha \left( \cdot \right)
}\left( \left \vert f\right \vert \right) (x) &\geq &\dint \limits_{B\left(
x,r_{x}\right) }\frac{\left \vert \Omega (x-y)\right \vert }{|x-y|^{n-\alpha
\left( x\right) }}\left \vert f(y)\right \vert dy \\
&\geq &\frac{C}{|B\left( x,r_{x}\right) |^{n-\alpha \left( x\right) }}\dint
\limits_{B\left( x,r_{x}\right) }\left \vert \Omega (x-y)\right \vert \left
\vert f(y)\right \vert dy.
\end{eqnarray*}
\end{proof}

From the process proving (\ref{5}) in \cite{Rafeiro}, it is easy to see that
the conclusions of (\ref{5}) also hold for $\widetilde{T}_{\left\vert \Omega
\right\vert ,\alpha \left( \cdot \right) }$. Combining this with (\ref{39}),
we can immediately obtain (\ref{6}), which completes the proof of Theorem %
\ref{teo3}.

\begin{remark}
Taking $\alpha \left( \cdot \right) =0$ in Lemma \ref{lemma100} and the
inequality%
\begin{equation*}
M_{\Omega }\left( f\right) (x)\leq C\widetilde{T}_{\left\vert \Omega
\right\vert }\left( \left\vert f\right\vert \right) (x)\qquad \text{for }%
x\in {\mathbb{R}^{n}},
\end{equation*}%
which follows from the definitions of the operators.
\end{remark}

The above theorems (Theorem \ref{teo1} and Theorem \ref{teo3}) allows to use
the known results for the boundedness of the operators $M_{\Omega ,\alpha
\left( \cdot \right) }$ and $I_{\Omega ,\alpha \left( \cdot \right) }$
transfer to the various function spaces. The following fact is known, see
Lemma 3.1. in \cite{Long}.

\begin{lemma}
\label{Lemma1}Let $E$ be a bounded open set, $p\left( x\right) \in \mathcal{P%
}^{\log }\left( E\right) $ and $\alpha \left( x\right) $ satisfy assumptions
(\ref{25}). Then,%
\begin{equation*}
\left\Vert \left\vert x-\cdot \right\vert ^{\alpha \left( x\right) -n}\chi _{%
\tilde{B}(x,r)}\right\Vert _{_{L^{p\left( \cdot \right) }\left( E\right)
}}\lesssim r^{\alpha \left( x\right) -\frac{n}{p\left( x\right) }}.
\end{equation*}
\end{lemma}

We will also make use of the estimate provided by the following fact (see 
\cite{Long}).%
\begin{equation}
\left\Vert \chi _{\tilde{B}(x,r)}\right\Vert _{L^{p\left( \cdot \right)
}\left( E\right) }\lesssim r^{\psi _{p}\left( x,r\right) },\qquad x\in E,%
\text{ }p\left( x\right) \in \mathcal{P}^{\log }\left( E\right) ,
\label{100}
\end{equation}%
where%
\begin{equation*}
\psi _{p}\left( x,r\right) =\left\{ 
\begin{array}{c}
\frac{n}{p\left( x\right) },\text{ }r\leq 1 \\ 
\frac{n}{p\left( \infty \right) },\text{ }r>1.%
\end{array}%
\right.
\end{equation*}

\subsection{\textbf{Preliminaries on variable exponent Morrey spaces }$%
L^{p\left( \cdot \right) ,\protect\lambda \left( \cdot \right) }$}

\begin{flushleft}
We define variable exponent Morrey space as follows.
\end{flushleft}

\begin{definition}
Let $E$ be a bounded open set and $\lambda \left( x\right) $ be a measurable
function on $E$ with values in $\left[ 0,n\right] $. Then, the variable
exponent Morrey space $L^{p\left( \cdot \right) ,\lambda \left( \cdot
\right) }\equiv L^{p\left( \cdot \right) ,\lambda \left( \cdot \right)
}\left( E\right) $ is defined by%
\begin{equation*}
L^{p\left( \cdot \right) ,\lambda \left( \cdot \right) }\equiv L^{p\left(
\cdot \right) ,\lambda \left( \cdot \right) }\left( E\right) =\left\{ 
\begin{array}{c}
f\in L_{loc}^{p\left( \cdot \right) }\left( E\right) : \\ 
\Vert f\Vert _{L^{p\left( \cdot \right) ,\lambda \left( \cdot \right)
}}=\sup\limits_{x\in E,r>0}r^{-\frac{\lambda \left( x\right) }{p\left(
x\right) }}\left\Vert f\chi _{\tilde{B}(x,r)}\right\Vert _{L^{p\left( \cdot
\right) }\left( E\right) }<\infty%
\end{array}%
\right\} .
\end{equation*}%
Note that $L^{p\left( \cdot \right) ,0}\left( E\right) =L^{p\left( \cdot
\right) }\left( E\right) $ and $L^{p\left( \cdot \right) ,n}\left( E\right)
=L^{\infty }\left( E\right) $. If $\lambda _{-}>n$, then $L^{p\left( \cdot
\right) ,\lambda \left( \cdot \right) }\left( E\right) =\left\{ 0\right\} $.
\end{definition}

\begin{lemma}
\label{Lemma2}Let $E$ be a bounded open set, $\Omega \in L_{s}(S^{n-1})$
with $1<s<\infty $, $p\left( x\right) ,q\left( x\right) \in \mathcal{P}%
^{\log }\left( E\right) $, $\alpha \left( x\right) $ satisfy the following
assumptions%
\begin{equation}
\left. 
\begin{array}{c}
\alpha _{0}=\limfunc{essinf}\limits_{x\in E}\alpha \left( x\right) >0 \\ 
\limfunc{esssup}\limits_{x\in E}\left[ \lambda \left( x\right) +\alpha
\left( x\right) p\left( x\right) \right] <n%
\end{array}%
\right\}  \label{11}
\end{equation}%
and $\left( p^{\prime }\right) _{+}\leq s$. Define $q\left( x\right) $ by $%
\frac{1}{q(\cdot )}=\frac{1}{p(\cdot )}-\frac{\alpha (\cdot )}{n-\lambda
(\cdot )}$. Then, the rough Riesz-type potential operator $I_{\Omega ,\alpha
\left( \cdot \right) }$ is $\left( L^{p\left( \cdot \right) ,\lambda \left(
\cdot \right) }\left( E\right) \rightarrow L^{q\left( \cdot \right) ,\lambda
\left( \cdot \right) }\left( E\right) \right) $-bounded. Moreover,%
\begin{equation*}
\left\Vert I_{\Omega ,\alpha \left( \cdot \right) }f\right\Vert _{L^{q\left(
\cdot \right) ,\lambda \left( \cdot \right) }\left( E\right) }\lesssim
\left\Vert f\right\Vert _{L^{p\left( \cdot \right) ,\lambda \left( \cdot
\right) }\left( E\right) }.
\end{equation*}
\end{lemma}

\begin{proof}
By the embedding property in Lemma 7 in \cite{Almeida}, we only need to
prove that the operator $I_{\Omega ,\alpha \left( \cdot \right) }$ is
bounded in $L^{p\left( \cdot \right) ,\lambda \left( \cdot \right) }\left(
E\right) $.

\textbf{Hedberg's trick:}

\begin{eqnarray}
I_{\Omega ,\alpha \left( \cdot \right) }f\left( x\right) 
&=&\int\limits_{B(x,2r)}\frac{\Omega (x-y)}{|x-y|^{n-\alpha \left( x\right) }%
}f(y)dy+\int\limits_{B^{C}(x,2r)}\frac{\Omega (x-y)}{|x-y|^{n-\alpha \left(
x\right) }}f(y)dy  \notag \\
&=&\mathcal{F}\left( x,r\right) +\mathcal{G}\left( x,r\right) .  \label{4.6}
\end{eqnarray}%
We may assume that $\left\Vert f\right\Vert _{L^{p\left( \cdot \right)
,\lambda \left( \cdot \right) }\left( E\right) }\leq 1$. For $\mathcal{F}%
\left( x,r\right) $, we first have to prove the following:%
\begin{equation}
\mathcal{F}\left( x,r\right) :=\left\vert \dint\limits_{\left\vert
x-y\right\vert <r}\frac{\Omega (x-y)}{|x-y|^{n-\alpha \left( x\right) }}%
f(y)dy\right\vert \lesssim \frac{2^{n}r^{\alpha \left( x\right) }}{2^{\alpha
\left( x\right) }-1}M_{\Omega }f\left( x\right) .  \label{4.2}
\end{equation}%
Indeed, for $f\left( x\right) \geq 0$ we have%
\begin{eqnarray*}
\mathcal{F}\left( x,r\right)  &=&\dsum\limits_{j=0}^{\infty
}\dint\limits_{2^{-j-1}r\leq |x-y|<2^{-j}r}\frac{\Omega (x-y)}{%
|x-y|^{n-\alpha \left( x\right) }}f(y)dy \\
&\leq &\dsum\limits_{j=0}^{\infty }\frac{1}{\left( 2^{-j-1}r\right)
^{n-\alpha \left( x\right) }}\dint\limits_{|x-y|<2^{-j}r}\Omega (x-y)f(y)dy
\\
&\leq &2^{n-\alpha \left( x\right) }M_{\Omega }f\left( x\right)
\dsum\limits_{j=0}^{\infty }\frac{\left\vert B\left( x,2^{-j}r\right)
\right\vert }{\left( 2^{-j}r\right) ^{n-\alpha \left( x\right) }}.
\end{eqnarray*}%
Hence by $\left\vert B\left( x,2^{-j}r\right) \right\vert \lesssim \left(
2^{-j}r\right) ^{n}$, we obtain%
\begin{equation*}
\mathcal{F}\left( x,r\right) \lesssim 2^{n-\alpha \left( x\right) }r^{\alpha
\left( x\right) }M_{\Omega }f\left( x\right) \dsum\limits_{j=0}^{\infty
}\left( 2^{-j\alpha \left( x\right) }\right) ,
\end{equation*}%
which gives the estimate (\ref{4.2}). Then by (\ref{4.2}):%
\begin{equation*}
\left\vert \mathcal{F}\left( x,r\right) \right\vert \lesssim r^{\alpha
\left( x\right) }M_{\Omega }f\left( x\right) .
\end{equation*}%
For $\mathcal{G}\left( x,r\right) $, from Lemma \ref{Lemma1} and the
procedure of Theorem 3 in \cite{Almeida}, we may show that%
\begin{equation*}
\left\vert \mathcal{G}\left( x,r\right) \right\vert \lesssim r^{\alpha
\left( x\right) -\frac{n-\lambda \left( x\right) }{p\left( x\right) }}.
\end{equation*}%
Then, from (\ref{4.6}) we get%
\begin{equation}
I_{\Omega ,\alpha \left( \cdot \right) }f\left( x\right) \lesssim \left[
r^{\alpha \left( x\right) }M_{\Omega }f\left( x\right) +r^{\alpha \left(
x\right) -\frac{n-\lambda \left( x\right) }{p\left( x\right) }}\right] .
\label{12}
\end{equation}%
As usual in Hedberg approach, we choose%
\begin{equation*}
r=\left[ M_{\Omega }f\left( x\right) \right] ^{-\frac{p\left( x\right) }{%
n-\lambda \left( x\right) }}.
\end{equation*}%
Substituting this into the (\ref{12}), we get%
\begin{equation*}
\left\vert I_{\Omega ,\alpha \left( \cdot \right) }f\left( x\right)
\right\vert \lesssim \left( M_{\Omega }f\left( x\right) \right) ^{\frac{%
p\left( x\right) }{q\left( x\right) }},
\end{equation*}%
here we need the (\ref{0*}). Therefore, by Theorem 5.1 in \cite{Rafeiro} we
know that%
\begin{equation*}
\dint\limits_{\tilde{B}\left( x,r\right) }\left\vert I_{\Omega ,\alpha
\left( \cdot \right) }f\left( y\right) \right\vert ^{q\left( y\right)
}dy\lesssim \dint\limits_{\tilde{B}\left( x,r\right) }\left\vert M_{\Omega
}f\left( y\right) \right\vert ^{p\left( y\right) }dy\lesssim r^{\lambda
\left( x\right) },
\end{equation*}%
which completes the proof of Lemma \ref{Lemma2}.
\end{proof}

\begin{theorem}
\label{teo4} Let $E$ be a bounded open set, $\Omega \in L_{s}(S^{n-1})$ with 
$1<s<\infty $, $p\left( x\right) ,q\left( x\right) \in \mathcal{P}^{\log
}\left( E\right) $, $\alpha \left( x\right) $ satisfy (\ref{11}) and $\left(
p^{\prime }\right) _{+}\leq s$. Define $q\left( x\right) $, $\mu \left(
x\right) $ by $\frac{1}{q(\cdot )}=\frac{1}{p(\cdot )}-\frac{\alpha (\cdot )%
}{n}$, $\frac{n-\mu (\cdot )}{q(\cdot )}=\frac{n-\lambda (\cdot )}{p(\cdot )}%
-\alpha (\cdot )$, respectively. Then, the rough Riesz-type potential
operator $I_{\Omega ,\alpha \left( \cdot \right) }$ is $\left( L^{p\left(
\cdot \right) ,\lambda \left( \cdot \right) }\left( E\right) \rightarrow
L^{q\left( \cdot \right) ,\mu \left( \cdot \right) }\left( E\right) \right) $%
-bounded. Moreover,%
\begin{equation}
\left\Vert I_{\Omega ,\alpha \left( \cdot \right) }f\right\Vert _{L^{q\left(
\cdot \right) ,\mu \left( \cdot \right) }\left( E\right) }\lesssim
\left\Vert f\right\Vert _{L^{p\left( \cdot \right) ,\lambda \left( \cdot
\right) }\left( E\right) },  \label{51}
\end{equation}%
where%
\begin{equation*}
1\leq q(\cdot )\leq \frac{p(\cdot )(n-\lambda (\cdot ))}{n-\lambda (\cdot
)-\alpha (\cdot )p(\cdot )}.
\end{equation*}
\end{theorem}

\begin{proof}
Since%
\begin{equation*}
\frac{p(\cdot )(n-\lambda (\cdot ))}{n-\lambda (\cdot )-\alpha (\cdot
)p(\cdot )}<\frac{np(\cdot )}{n-\alpha (\cdot )p(\cdot )},
\end{equation*}%
from Lemma \ref{Lemma2} and Theorem \ref{teo1} we obtain%
\begin{eqnarray*}
\left\Vert I_{\Omega ,\alpha \left( \cdot \right) }f\right\Vert _{L^{q\left(
\cdot \right) ,\mu \left( \cdot \right) }\left( E\right) }
&=&\sup\limits_{x\in E,r>0}r^{-\frac{\mu \left( x\right) }{q\left( x\right) }%
}\left\Vert I_{\Omega ,\alpha \left( \cdot \right) }f\chi _{\tilde{B}%
(x,r)}\right\Vert _{L^{q\left( \cdot \right) }\left( E\right) } \\
&\lesssim &\sup\limits_{x\in E,r>0}r^{-\frac{\lambda \left( x\right) }{%
p\left( x\right) }}\left\Vert f\chi _{\tilde{B}(x,r)}\right\Vert
_{L^{p\left( \cdot \right) }\left( E\right) } \\
&=&\left\Vert f\right\Vert _{L^{p\left( \cdot \right) ,\lambda \left( \cdot
\right) }\left( E\right) }.
\end{eqnarray*}%
Clearly, Theorem \ref{teo4} holds.
\end{proof}

\begin{theorem}
Under the conditions of Theorem \ref{teo4},%
\begin{equation}
\left\Vert M_{\Omega ,\alpha \left( \cdot \right) }f\right\Vert _{L^{q\left(
\cdot \right) ,\mu \left( \cdot \right) }\left( E\right) }\lesssim
\left\Vert f\right\Vert _{L^{p\left( \cdot \right) ,\lambda \left( \cdot
\right) }\left( E\right) }.  \label{52}
\end{equation}
\end{theorem}

\begin{proof}
Similar to the proof of Theorem \ref{teo3}, the conclusion (\ref{52}) is a
direct consequence of (\ref{39}) and (\ref{51}). Indeed, from the process
proving (\ref{51}) in Theorem \ref{teo4}, it is easy to see that the
conclusion (\ref{51}) also holds for $\widetilde{T}_{\left\vert \Omega
\right\vert ,\alpha \left( \cdot \right) }$. Combining this with (\ref{39}),
we can immediately obtain (\ref{52}), which completes the proof.
\end{proof}

\subsection{Preliminaries on \textbf{variable exponent} vanishing
generalized \textbf{Morrey spaces}}

\begin{flushleft}
In this section we first consider the generalized Morrey spaces $L^{p\left(
\cdot \right) ,w\left( \cdot \right) }\left( E\right) $ with variable
exponent $p(x)$ and a general function $w(x,r):\Pi \times \left(
0,diam\left( E\right) \right) \rightarrow 
\mathbb{R}
_{+}$, $\Pi \subset E\subset 
\mathbb{R}
^{n}$, defining the Morrey type norm on sets $E\subset 
\mathbb{R}
^{n}$ which may be both bounded and unbounded; see the definition of the
spaces $L^{p\left( \cdot \right) ,w\left( \cdot \right) }\left( E\right) $
in (\ref{44}) below.
\end{flushleft}

Everywhere in the sequel the functions $w\left( x,r\right) $, $w_{1}\left(
x,r\right) $, $w_{2}\left( x,r\right) $ used in the body of this paper, are
non-negative measurable functions on $E\times \left( 0,\infty \right) $,
where $E\subset 
\mathbb{R}
^{n}$ is an open set. We recall the definition of variable exponent
generalized Morrey space in the following.

\begin{definition}
Let $1\leq p\left( x\right) \leq p_{+}<\infty $, $\Pi \subset E\subset 
\mathbb{R}
^{n}$, $x\in \Pi $, $w(x,r):\Pi \times \left( 0,diam\left( E\right) \right)
\rightarrow 
\mathbb{R}
_{+}$, where%
\begin{equation}
\inf\limits_{x\in \Pi }w(x,r)>0\qquad r>0.  \label{43}
\end{equation}%
Then, the variable exponent generalized Morrey space $L_{\Pi }^{p\left(
\cdot \right) ,w\left( \cdot \right) }\equiv L_{\Pi }^{p\left( \cdot \right)
,w\left( \cdot \right) }\left( E\right) $ is defined by%
\begin{equation}
L_{\Pi }^{p\left( \cdot \right) ,w\left( \cdot \right) }\equiv L_{\Pi
}^{p\left( \cdot \right) ,w\left( \cdot \right) }\left( E\right) =\left\{ 
\begin{array}{c}
f\in L_{loc}^{p\left( \cdot \right) }\left( E\right) : \\ 
\Vert f\Vert _{L_{\Pi }^{p\left( \cdot \right) ,w\left( \cdot \right)
}}=\sup\limits_{x\in \Pi ,r>0}w(x,r)^{-\frac{1}{p\left( x\right) }%
}\left\Vert f\right\Vert _{L^{p\left( \cdot \right) }\left( \tilde{B}%
(x,r)\right) }<\infty%
\end{array}%
\right\} ,  \label{44}
\end{equation}%
and one can also see that for bounded exponents $p$ there holds the
following equivalence:%
\begin{equation*}
f\in L_{\Pi }^{p\left( \cdot \right) ,w\left( \cdot \right) }\text{ if and
only if }\sup_{x\in \Pi ,r>0}\dint\limits_{\tilde{B}(x,r)}\left\vert \frac{%
f\left( y\right) }{w\left( x,r\right) }\right\vert ^{p\left( y\right)
}dy<\infty .
\end{equation*}
\end{definition}

On the other hand, the above definition recover the definition of $%
L^{p\left( \cdot \right) ,\lambda \left( \cdot \right) }\left( E\right) $ if
we choose $w(x,r)=r^{\frac{\lambda \left( x\right) }{p\left( x\right) }}$
and $\Pi =E$, that is 
\begin{equation*}
L^{p\left( \cdot \right) ,\lambda \left( \cdot \right) }\left( E\right)
=L_{\Pi }^{p\left( \cdot \right) ,w\left( \cdot \right) }\left( E\right)
\mid _{w(x,r)=r^{\frac{\lambda \left( x\right) }{p\left( x\right) }}}.
\end{equation*}%
Also, when $\Pi =\left\{ x_{0}\right\} $ and $\Pi =E$, $L_{\Pi }^{p\left(
\cdot \right) ,w\left( \cdot \right) }$ turns into the local generalized
Morrey space $L_{\left\{ x_{0}\right\} }^{p(\cdot ),w\left( \cdot \right)
}(E)$ and the global generalized Morrey space $L_{E}^{p(\cdot ),w\left(
\cdot \right) }(E)$ , respectively. Moreover, we point out that $w(x,r)$ is
a measurable non-negative function and no monotonicity type condition is
imposed on these spaces. Note that by the above definition of the norm in $%
L^{p\left( \cdot \right) }\left( E\right) $ (see \ref{0}), we can also write
that%
\begin{equation*}
\Vert f\Vert _{L_{\Pi }^{p\left( \cdot \right) ,w\left( \cdot \right)
}}=\sup_{x\in \Pi ,r>0}\inf \left\{ \lambda =\lambda \left( x,r\right)
:\dint\limits_{\tilde{B}(x,r)}\left\vert \frac{f\left( y\right) }{\lambda
w\left( x,r\right) }\right\vert ^{p\left( y\right) }dy\leq 1\right\} .
\end{equation*}

Then, recall that the concept of the variable exponent vanishing generalized
Morrey space $VL_{\Pi }^{p\left( \cdot \right) ,w\left( \cdot \right)
}\left( E\right) $ has been introduced in \cite{Long} in the following form:

\begin{definition}
\label{definition}Let $1\leq p\left( x\right) \leq p_{+}<\infty $, $\Pi
\subset E\subset 
\mathbb{R}
^{n}$, $x\in \Pi $, $w(x,r):\Pi \times \left( 0,diam\left( E\right) \right)
\rightarrow 
\mathbb{R}
_{+}$. Then, the variable exponent vanishing generalized Morrey space $%
VL_{\Pi }^{p\left( \cdot \right) ,w\left( \cdot \right) }\equiv VL_{\Pi
}^{p\left( \cdot \right) ,w\left( \cdot \right) }\left( E\right) $ is
defined by%
\begin{equation*}
\left\{ f\in L_{\Pi }^{p\left( \cdot \right) ,w\left( \cdot \right) }\left(
E\right) :\lim\limits_{r\rightarrow 0}\sup\limits_{x\in \Pi }\mathfrak{M}%
_{p\left( \cdot \right) ,w\left( \cdot \right) }\left( f;x,r\right)
=0\right\} ,
\end{equation*}%
where 
\begin{equation*}
\mathfrak{M}_{p\left( \cdot \right) ,w\left( \cdot \right) }\left(
f;x,r\right) :=\frac{r^{-\frac{n}{p\left( x\right) }}\left\Vert f\right\Vert
_{L^{p\left( \cdot \right) }\left( \tilde{B}(x,r)\right) }}{w(x,r)^{\frac{1}{%
p\left( x\right) }}}.
\end{equation*}%
Naturally, it is suitable to impose on $w(x,t)$ with the following
conditions:%
\begin{equation}
\lim_{t\rightarrow 0}\sup\limits_{x\in \Pi }\frac{t^{-\psi _{p}\left(
x,t\right) }}{w(x,t)^{\frac{1}{p\left( x\right) }}}=0  \label{2*}
\end{equation}%
and%
\begin{equation}
\inf_{t>1}\sup\limits_{x\in \Pi }w(x,t)>0.  \label{3*}
\end{equation}
\end{definition}

From (\ref{2*}) and (\ref{3*}), we easily know that the bounded functions
with compact support belong to $VL_{\Pi }^{p\left( \cdot \right) ,w\left(
\cdot \right) }\left( E\right) $, which make the spaces $VL_{\Pi }^{p\left(
\cdot \right) ,w\left( \cdot \right) }\left( E\right) $ non-trivial.

The spaces $VL_{\Pi }^{p\left( \cdot \right) ,w\left( \cdot \right) }\left(
E\right) $ are Banach spaces with respect to the norm 
\begin{equation*}
\Vert f\Vert _{VL_{\Pi }^{p\left( \cdot \right) ,w\left( \cdot \right)
}}\equiv \Vert f\Vert _{L_{\Pi }^{p\left( \cdot \right) ,w\left( \cdot
\right) }}=\sup\limits_{x\in \Pi ,r>0}\mathfrak{M}_{p\left( \cdot \right)
,w\left( \cdot \right) }\left( f;x,r\right) .
\end{equation*}%
The spaces $VL_{\Pi }^{p\left( \cdot \right) ,w\left( \cdot \right) }\left(
E\right) $ are also closed subspaces of the Banach spaces $L_{\Pi }^{p\left(
\cdot \right) ,w\left( \cdot \right) }\left( E\right) $, which may be shown
by standard means.

Furthermore, we have the following embeddings:%
\begin{equation*}
VL_{\Pi }^{p\left( \cdot \right) ,w\left( \cdot \right) }\subset L_{\Pi
}^{p\left( \cdot \right) ,w\left( \cdot \right) },\qquad \Vert f\Vert
_{L_{\Pi }^{p\left( \cdot \right) ,w\left( \cdot \right) }}\leq \Vert f\Vert
_{VL_{\Pi }^{p\left( \cdot \right) ,w\left( \cdot \right) }}.
\end{equation*}

In 2016, for bounded or unbounded sets $E$, Long and Han \cite{Long}
considered the Spanne type boundedness of operators $M_{\alpha \left( \cdot
\right) }$ and $I_{\alpha \left( \cdot \right) }$ on $VL_{\Pi }^{p\left(
\cdot \right) ,w\left( \cdot \right) }\left( E\right) $.

Now, in this section we extend Theorem 4.3. in \cite{Long} to rough kernel
versions. In other words, the Theorem 4.3. in \cite{Long} allows to use the
known results for the boundedness of the operators $I_{\alpha \left( \cdot
\right) }$ and $M_{\alpha \left( \cdot \right) }$ in generalized variable
exponent Morrey spaces to transfer them to the operators $I_{\Omega ,\alpha
\left( \cdot \right) }$ and $M_{\Omega ,\alpha \left( \cdot \right) }$. We
give two versions of such an extension, the one being a generalization of
Spanne's result for rough potential operators with variable order, the other
extending the corresponding Adams' result, respectively.

In this context, we will give some answers to the above explanations as
follows:

\begin{theorem}
\label{teo7}\textbf{(Spanne type result with variable }$\alpha \left(
x\right) $\textbf{) }(our main result) Let $E$ be a bounded open set, $%
\Omega \in L_{s}(S^{n-1})$, $1<s\leq \infty $, $\Omega (\mu x)=\Omega (x)~$%
for any$~\mu >0$, $x\in {\mathbb{R}^{n}}\setminus \{0\}$, $p\left( x\right)
\in \mathcal{P}^{\log }\left( E\right) $, $\alpha \left( x\right) $ satisfy
the assumption (\ref{25}). Define $q\left( x\right) $ by (\ref{26}). Suppose
that $q\left( \cdot \right) $ and $\alpha \left( \cdot \right) $ satisfy (%
\ref{1}). For $\frac{s}{s-1}<p^{-}\leq p\left( \cdot \right) <\frac{n}{%
\alpha \left( \cdot \right) }$, the following pointwise estimate 
\begin{equation}
\left\Vert I_{\Omega ,\alpha \left( \cdot \right) }f\right\Vert _{L^{q\left(
\cdot \right) }\left( \tilde{B}(x,r)\right) }\lesssim r^{\frac{n}{q\left(
x\right) }}\int\limits_{r}^{diam\left( E\right) }\left\Vert f\right\Vert
_{L^{p\left( \cdot \right) }\left( \tilde{B}(x,t)\right) }\frac{dt}{t^{\frac{%
n}{q\left( x\right) }+1}}  \label{40}
\end{equation}%
holds for any ball $\tilde{B}(x,r)$ and for all $f\in L_{loc}^{p\left( \cdot
\right) }\left( E\right) $.

If the functions $w_{1}\left( x,r\right) $ and $w_{2}\left( x,r\right) $
satisfy (\ref{43}) as well as the following Zygmund condition%
\begin{equation}
\int\limits_{r}^{diam\left( E\right) }\frac{w_{1}^{\frac{1}{p\left( x\right) 
}}(x,t)}{t^{1-\alpha \left( x\right) }}dt\lesssim \,w_{2}^{\frac{1}{q\left(
x\right) }}(x,r),\qquad r\in \left( 0,diam\left( E\right) \right]
\label{316}
\end{equation}%
and additionally these functions satisfy the conditions (\ref{2*})-(\ref{3*}%
), 
\begin{equation}
c_{\delta }:=\dint\limits_{\delta }^{diam\left( E\right) }\sup_{x\in \Pi }%
\frac{w_{1}^{\frac{1}{p\left( x\right) }}(x,t)}{t^{1-\alpha \left( x\right) }%
}dt<\infty ,\qquad \delta >0  \label{1*}
\end{equation}%
then the operators $I_{\Omega ,\alpha \left( \cdot \right) }$ and $M_{\Omega
,\alpha \left( \cdot \right) }$ are $\left( VL_{\Pi }^{p\left( \cdot \right)
,w_{1}\left( \cdot \right) }\left( E\right) \rightarrow VL_{\Pi }^{q\left(
\cdot \right) ,w_{2}\left( \cdot \right) }\left( E\right) \right) $-bounded.
Moreover,%
\begin{equation}
\left\Vert I_{\Omega ,\alpha \left( \cdot \right) }f\right\Vert _{VL_{\Pi
}^{q\left( \cdot \right) ,w_{2}\left( \cdot \right) }\left( E\right)
}\lesssim \left\Vert f\right\Vert _{VL_{\Pi }^{p\left( \cdot \right)
,w_{1}\left( \cdot \right) }\left( E\right) },  \label{41}
\end{equation}%
\begin{equation*}
\left\Vert M_{\Omega ,\alpha \left( \cdot \right) }f\right\Vert _{VL_{\Pi
}^{q\left( \cdot \right) ,w_{2}\left( \cdot \right) }\left( E\right)
}\lesssim \left\Vert f\right\Vert _{VL_{\Pi }^{p\left( \cdot \right)
,w_{1}\left( \cdot \right) }\left( E\right) }.
\end{equation*}
\end{theorem}

\begin{proof}
Since inequality (\ref{40}) is the key of the proof of (\ref{41}), we first
prove (\ref{40}).

For any $x\in E$, we write as 
\begin{equation}
f\left( y\right) =f_{1}\left( y\right) +f_{2}\left( y\right) ,  \label{42}
\end{equation}%
where $f_{1}\left( y\right) =f\left( y\right) \chi _{\tilde{B}\left(
x,2r\right) }\left( y\right) $, $r>0$ such that%
\begin{equation*}
I_{\Omega ,\alpha \left( \cdot \right) }f\left( y\right) =I_{\Omega ,\alpha
\left( \cdot \right) }f_{1}\left( y\right) +I_{\Omega ,\alpha \left( \cdot
\right) }f_{2}\left( y\right) .
\end{equation*}%
By using triangle inequality, we get%
\begin{equation*}
\left\Vert I_{\Omega ,\alpha \left( \cdot \right) }f\right\Vert _{L^{q\left(
\cdot \right) }\left( \tilde{B}\left( x,r\right) \right) }\leq \left\Vert
I_{\Omega ,\alpha \left( \cdot \right) }f_{1}\right\Vert _{L^{q\left( \cdot
\right) }\left( \tilde{B}\left( x,r\right) \right) }+\left\Vert I_{\Omega
,\alpha \left( \cdot \right) }f_{2}\right\Vert _{L^{q\left( \cdot \right)
}\left( \tilde{B}\left( x,r\right) \right) }.
\end{equation*}%
Now, let us estimate $\left\Vert I_{\Omega ,\alpha \left( \cdot \right)
}f_{1}\right\Vert _{L^{q\left( \cdot \right) }\left( \tilde{B}\left(
x,r\right) \right) }$ and $\left\Vert I_{\Omega ,\alpha \left( \cdot \right)
}f_{2}\right\Vert _{L^{q\left( \cdot \right) }\left( \tilde{B}\left(
x,r\right) \right) }$, respectively.

By Hardy-Littlewood-Sobolev type inequality and Theorem \ref{teo1}, we
obtain that%
\begin{eqnarray*}
\left\Vert I_{\Omega ,\alpha \left( \cdot \right) }f_{1}\right\Vert
_{L^{q\left( \cdot \right) }\left( \tilde{B}\left( x,r\right) \right) }
&\leq &\left\Vert I_{\Omega ,\alpha \left( \cdot \right) }f_{1}\right\Vert
_{L^{q\left( \cdot \right) }\left( E\right) }\lesssim \left\Vert
f_{1}\right\Vert _{L^{p\left( \cdot \right) }\left( E\right) }=\left\Vert
f\right\Vert _{L^{p\left( \cdot \right) }\left( \tilde{B}\left( x,2r\right)
\right) } \\
&\approx &r^{\frac{n}{q\left( x\right) }}\left\Vert f\right\Vert
_{L^{p\left( \cdot \right) }\left( \tilde{B}\left( x,2r\right) \right)
}\dint\limits_{2r}^{diam\left( E\right) }\frac{dt}{t^{\frac{n}{q\left(
x\right) }+1}} \\
&\leq &r^{\frac{n}{q\left( x\right) }}\dint\limits_{r}^{diam\left( E\right)
}\left\Vert f\right\Vert _{L^{p\left( \cdot \right) }\left( \tilde{B}\left(
x,t\right) \right) }\frac{dt}{t^{\frac{n}{q\left( x\right) }+1}},
\end{eqnarray*}%
where in the last inequality, we have used the following fact:%
\begin{equation*}
\left\Vert f\right\Vert _{L^{p\left( \cdot \right) }\left( \tilde{B}\left(
x,2r\right) \right) }\leq \left\Vert f\right\Vert _{L^{p\left( \cdot \right)
}\left( \tilde{B}\left( x,t\right) \right) }\text{, for }t>2r.
\end{equation*}

Now, let us estimate the second part. For the estimate used in $\left\Vert
I_{\Omega ,\alpha \left( \cdot \right) }f_{2}\right\Vert _{L^{q\left( \cdot
\right) }\left( \tilde{B}\left( x,r\right) \right) }$, we first have to
prove the below inequality:%
\begin{equation}
\left\vert I_{\Omega ,\alpha \left( \cdot \right) }f_{2}\left( x\right)
\right\vert \lesssim \left\Vert \Omega \right\Vert _{L^{q\left( \cdot
\right) }\left( S^{n-1}\right) }\int\limits_{2r}^{diam\left( E\right)
}\left\Vert f\right\Vert _{L^{p\left( \cdot \right) }\left( \tilde{B}\left(
x,t\right) \right) }\frac{dt}{t^{\frac{n}{q\left( x\right) }+1}}.  \label{10}
\end{equation}%
Indeed, if $\left\vert x-z\right\vert \leq r$ and $\left\vert z-y\right\vert
\geq r$, then $\left\vert x-y\right\vert \leq \left\vert x-z\right\vert
+\left\vert y-z\right\vert \leq 2\left\vert y-z\right\vert $. By generalized
Minkowski's inequality we get%
\begin{eqnarray*}
\left\Vert I_{\Omega ,\alpha \left( \cdot \right) }f_{2}\right\Vert
_{L^{q\left( \cdot \right) }\left( \tilde{B}\left( x,r\right) \right) }
&=&\left\Vert \dint\limits_{E\setminus \tilde{B}\left( x,2r\right) }\frac{%
\Omega (z-y)}{|z-y|^{n-\alpha \left( x\right) }}f(y)dy\right\Vert
_{L^{q\left( \cdot \right) }\left( \tilde{B}\left( x,r\right) \right) } \\
&\lesssim &\dint\limits_{E\setminus \tilde{B}\left( x,2r\right) }\frac{%
\left\vert \Omega (z-y)\right\vert \left\vert f(y)\right\vert }{%
|x-y|^{n-\alpha \left( x\right) }}dy\left\Vert \chi _{\tilde{B}\left(
x,r\right) }\right\Vert _{L^{q\left( \cdot \right) }\left( E\right) }.
\end{eqnarray*}

Put $\gamma >\frac{n}{q\left( \cdot \right) }$. Provided that $1<s^{\prime
}<p^{-}\leq p^{+}<\infty $, $\sup\limits_{x\in E}\left( \alpha \left(
x\right) +\gamma -n\right) <\infty $ and $\inf\limits_{x\in E}\left(
n+\left( \alpha \left( x\right) +\gamma -n\right) \left( \frac{p\left( \cdot
\right) }{s^{\prime }}\right) ^{\prime }\right) <\infty $, by generalized H%
\"{o}lder's inequality for $L^{p\left( \cdot \right) }\left( E\right) $,
Fubini's theorem and Lemma \ref{Lemma1} and (\ref{0*}), we obtain%
\begin{eqnarray}
&&\dint\limits_{E\setminus \tilde{B}\left( x,2r\right) }\frac{\left\vert
\Omega (z-y)\right\vert \left\vert f(y)\right\vert }{|x-y|^{n-\alpha \left(
x\right) }}dy  \notag \\
&\lesssim &\dint\limits_{E\setminus \tilde{B}\left( x,2r\right) }\frac{%
\left\vert \Omega (z-y)\right\vert \left\vert f(y)\right\vert }{%
|x-y|^{n-\alpha \left( x\right) -\gamma }}dy\dint\limits_{\left\vert
x-y\right\vert }^{diam\left( E\right) }\frac{dt}{t^{\gamma +1}}  \notag \\
&=&\dint\limits_{2r}^{diam\left( E\right) }\frac{dt}{t^{\gamma +1}}%
\dint\limits_{\left\{ y\in E:2r\leq |x-y|\leq t\right\} }\frac{\left\vert
\Omega (z-y)\right\vert \left\vert f(y)\right\vert }{|x-y|^{n-\alpha \left(
x\right) -\gamma }}dy  \notag \\
&\lesssim &\dint\limits_{2r}^{diam\left( E\right) }\left\Vert f\right\Vert
_{L^{p\left( \cdot \right) }\left( \tilde{B}\left( x,t\right) \right)
}\left\Vert x-\cdot |^{\alpha \left( x\right) +\gamma -n}\right\Vert
_{L^{\nu \left( \cdot \right) }\left( \tilde{B}\left( x,t\right) \right)
}\left\Vert \Omega \left( z-y\right) \right\Vert _{L_{s}\left( \tilde{B}%
\left( x,t\right) \right) }\frac{dt}{t^{\gamma +1}}  \notag
\end{eqnarray}%
\begin{equation}
\lesssim \dint\limits_{r}^{diam\left( E\right) }\left\Vert f\right\Vert
_{L^{p\left( \cdot \right) }\left( \tilde{B}\left( x,t\right) \right) }\frac{%
dt}{t^{\frac{n}{q\left( x\right) }+1}}  \label{13*}
\end{equation}%
for $\frac{1}{p\left( \cdot \right) }+\frac{1}{s}+\frac{1}{\nu \left( \cdot
\right) }=1$. Thus, by (\ref{100}) we get%
\begin{equation*}
\left\Vert I_{\Omega ,\alpha \left( \cdot \right) }f_{2}\right\Vert
_{L^{q\left( \cdot \right) }\left( \tilde{B}\left( x,r\right) \right)
}\lesssim r^{\frac{n}{q\left( x\right) }}\dint\limits_{r}^{diam\left(
E\right) }\left\Vert f\right\Vert _{L^{p\left( \cdot \right) }\left( \tilde{B%
}\left( x,t\right) \right) }\frac{dt}{t^{\frac{n}{q\left( x\right) }+1}}.
\end{equation*}%
Combining all the estimates for $\left\Vert I_{\Omega ,\alpha \left( \cdot
\right) }f_{1}\right\Vert _{L^{q\left( \cdot \right) }\left( \tilde{B}\left(
x,r\right) \right) }$ and $\left\Vert I_{\Omega ,\alpha \left( \cdot \right)
}f_{2}\right\Vert _{L^{q\left( \cdot \right) }\left( \tilde{B}\left(
x,r\right) \right) }$, we get (\ref{40}).

At last, by Definition \ref{definition}, (\ref{40}) and (\ref{316}) we get%
\begin{eqnarray*}
\left\Vert I_{\Omega ,\alpha \left( \cdot \right) }f\right\Vert _{VL_{\Pi
}^{q\left( \cdot \right) ,w_{2}\left( \cdot \right) }\left( E\right) }
&=&\sup\limits_{x\in \Pi ,r>0}\frac{r^{-\frac{n}{q\left( x\right) }%
}\left\Vert I_{\Omega ,\alpha \left( \cdot \right) }f\right\Vert
_{L^{q\left( \cdot \right) }\left( \tilde{B}(x,r)\right) }}{w_{2}(x,r)^{%
\frac{1}{q\left( x\right) }}} \\
&\lesssim &\sup\limits_{x\in \Pi ,r>0}\frac{1}{w_{2}(x,r)^{\frac{1}{q\left(
x\right) }}}\dint\limits_{r}^{diam\left( E\right) }\left\Vert f\right\Vert
_{L^{p\left( \cdot \right) }\left( \tilde{B}\left( x,t\right) \right) }\frac{%
dt}{t^{\frac{n}{q\left( x\right) }+1}} \\
&\lesssim &\left\Vert f\right\Vert _{VL_{\Pi }^{p\left( \cdot \right)
,w_{1}\left( \cdot \right) }\left( E\right) }\sup\limits_{x\in \Pi ,r>0}%
\frac{1}{w_{2}(x,r)^{\frac{1}{q\left( x\right) }}}\int\limits_{r}^{diam%
\left( E\right) }\frac{w_{1}^{\frac{1}{p\left( x\right) }}(x,t)}{t^{1-\alpha
\left( x\right) }}dt \\
&\lesssim &\left\Vert f\right\Vert _{VL_{\Pi }^{p\left( \cdot \right)
,w_{1}\left( \cdot \right) }\left( E\right) }
\end{eqnarray*}%
and 
\begin{equation*}
\lim_{r\rightarrow 0}\sup\limits_{x\in \Pi }\frac{r^{-\frac{n}{q\left(
x\right) }}\left\Vert I_{\Omega ,\alpha \left( \cdot \right) }f\right\Vert
_{L^{q\left( \cdot \right) }\left( \tilde{B}(x,r)\right) }}{w_{2}(x,t)^{%
\frac{1}{q\left( x\right) }}}\lesssim \lim_{r\rightarrow 0}\sup\limits_{x\in
\Pi }\frac{r^{-\frac{n}{p\left( x\right) }}\left\Vert f\right\Vert
_{L^{p\left( \cdot \right) }\left( \tilde{B}(x,r)\right) }}{w_{1}(x,t)^{%
\frac{1}{p\left( x\right) }}}=0.
\end{equation*}%
Thus, (\ref{41}) holds. On the other hand, since $M_{\Omega ,\alpha \left(
\cdot \right) }\left( f\right) \lesssim I_{\left\vert \Omega \right\vert
,\alpha \left( \cdot \right) }\left( \left\vert f\right\vert \right) $ (see
Lemma \ref{lemma100}) we can also use the same method for $M_{\Omega ,\alpha
\left( \cdot \right) }$, so we omit the details. As a result, we complete
the proof of Theorem \ref{teo7}.
\end{proof}

\begin{definition}
\textbf{(Rough }$\left( p,q\right) $\textbf{-admissible }$T_{\Omega ,\alpha
\left( \cdot \right) }$\textbf{-potential type operator with variable order)}
Let $1\leq p_{-}\left( E\right) \leq p\left( \cdot \right) \leq p_{+}\left(
E\right) <\infty $. A rough sublinear operator with variable order $%
T_{\Omega ,\alpha \left( \cdot \right) }$, i.e. $\left\vert T_{\Omega
,\alpha \left( \cdot \right) }\left( f+g\right) \right\vert \leq \left\vert
T_{\Omega ,\alpha \left( \cdot \right) }\left( f\right) \right\vert
+\left\vert T_{\Omega ,\alpha \left( \cdot \right) }\left( g\right)
\right\vert $ and for $\forall \lambda \in 
\mathbb{C}
$ $\left\vert T_{\Omega ,\alpha \left( \cdot \right) }\left( \lambda
f\right) \right\vert =\left\vert \lambda \right\vert \left\vert T_{\Omega
,\alpha \left( \cdot \right) }\left( f\right) \right\vert $, will be called
rough $\left( p,q\right) $-admissible $T_{\Omega ,\alpha \left( \cdot
\right) }$-potential type operator with variable order if

$\cdot $ $T_{\Omega ,\alpha \left( \cdot \right) }$ fullfills the following
size condition:%
\begin{equation}
\chi _{B\left( z,r\right) }\left( x\right) |T_{\Omega ,\alpha \left( \cdot
\right) }\left( f\chi _{E\diagdown B\left( z,2r\right) }\right) (x)|\leq
C\chi _{B\left( z,r\right) }\left( x\right) \int\limits_{E\diagdown B\left(
z,2r\right) }\frac{|\Omega (x-y)|}{|x-y|^{n-\alpha \left( \cdot \right) }}%
\,|f(y)|\,dy,  \label{e1}
\end{equation}

$\cdot $ $T_{\Omega ,\alpha \left( \cdot \right) }$ is $\left( L^{p\left(
\cdot \right) }\left( E\right) \rightarrow L^{q\left( \cdot \right) }\left(
E\right) \right) $-bounded.
\end{definition}

\begin{remark}
Note that rough $\left( p,q\right) $-admissible potential type operators
were introduced to study their boundedness on Morrey spaces with variable
exponents in \cite{Ho}. The operators $M_{\Omega ,\alpha \left( \cdot
\right) }$ and $I_{\Omega ,\alpha \left( \cdot \right) }$ are also rough $%
\left( p,q\right) $-admissible potential type operators. Moreover, these
operators satisfy (\ref{e1}).
\end{remark}

\begin{corollary}
Obviously, under the conditions of Theorem \ref{teo7}, if the rough $\left(
p,q\right) $-admissible $T_{\Omega ,\alpha \left( \cdot \right) }$-potential
type operator is $\left( L^{p\left( \cdot \right) }\left( E\right)
\rightarrow L^{q\left( \cdot \right) }\left( E\right) \right) $-bounded and
satisfies (\ref{e1}), the result in Theorem \ref{teo7} still holds.
\end{corollary}

For $\alpha \left( x\right) =0$ in Theorem \ref{teo7}, we get the following
new result:

\begin{corollary}
Let $E$, $\Omega $, $p\left( x\right) $ be the same as in Theorem \ref{teo7}%
. Then, for $\frac{s}{s-1}<p^{-}\leq p\left( \cdot \right) \leq p^{+}<\infty 
$, the following pointwise estimate 
\begin{equation*}
\left\Vert T_{\Omega }f\right\Vert _{L^{p\left( \cdot \right) }\left( \tilde{%
B}(x,r)\right) }\lesssim r^{\frac{n}{p\left( x\right) }}\int%
\limits_{r}^{diam\left( E\right) }t^{-\frac{n}{p\left( x\right) }%
-1}\left\Vert f\right\Vert _{L^{p\left( \cdot \right) }\left( \tilde{B}%
(x,t)\right) }dt
\end{equation*}%
holds for any ball $\tilde{B}(x,r)$ and for all $f\in L_{loc}^{p\left( \cdot
\right) }\left( E\right) $.

If the function $w\left( x,r\right) $ satisfies (\ref{43}) as well as the
following Zygmund condition%
\begin{equation*}
\int\limits_{r}^{diam\left( E\right) }\frac{w^{\frac{1}{p\left( x\right) }%
}(x,t)}{t}dt\lesssim \,w^{\frac{1}{p\left( x\right) }}(x,r),\qquad r\in
\left( 0,diam\left( E\right) \right]
\end{equation*}%
and additionally this function satisfies the conditions (\ref{2*})-(\ref{3*}%
), 
\begin{equation*}
c_{\delta }:=\dint\limits_{\delta }^{diam\left( E\right) }\sup_{x\in \Pi }%
\frac{w^{\frac{1}{p\left( x\right) }}(x,t)}{t}dt<\infty ,\qquad \delta >0
\end{equation*}%
then the operators $T_{\Omega }$ and $M_{\Omega }$ are bounded on $VL_{\Pi
}^{p\left( \cdot \right) ,w\left( \cdot \right) }\left( E\right) $. Moreover,%
\begin{equation*}
\left\Vert T_{\Omega }f\right\Vert _{VL_{\Pi }^{p\left( \cdot \right)
,w\left( \cdot \right) }\left( E\right) }\lesssim \left\Vert f\right\Vert
_{VL_{\Pi }^{p\left( \cdot \right) ,w\left( \cdot \right) }\left( E\right) },
\end{equation*}%
\begin{equation}
\left\Vert M_{\Omega }f\right\Vert _{VL_{\Pi }^{p\left( \cdot \right)
,w\left( \cdot \right) }\left( E\right) }\lesssim \left\Vert f\right\Vert
_{VL_{\Pi }^{p\left( \cdot \right) ,w\left( \cdot \right) }\left( E\right) }.
\label{20}
\end{equation}
\end{corollary}

\begin{theorem}
\label{teo8}\textbf{(Adams type result with variable }$\alpha \left(
x\right) $\textbf{) }(our main result) Let $E$, $\Omega $, $p\left( x\right) 
$, $q\left( x\right) $, $\alpha \left( x\right) $ be the same as in Theorem %
\ref{teo7}. Then, for $\frac{s}{s-1}<p^{-}\leq p\left( \cdot \right) <\frac{n%
}{\alpha \left( \cdot \right) }$ , the following pointwise estimate 
\begin{equation}
\left\vert I_{\Omega ,\alpha \left( \cdot \right) }f\left( x\right)
\right\vert \lesssim r^{\alpha \left( x\right) }M_{\Omega }f\left( x\right)
+\int\limits_{r}^{diam\left( E\right) }t^{\alpha \left( x\right) -\frac{n}{%
p\left( x\right) }-1}\left\Vert f\right\Vert _{L_{p}\left( \tilde{B}%
(x,t)\right) }dt  \label{15}
\end{equation}%
holds for any ball $\tilde{B}(x,r)$ and for all $f\in L_{loc}^{p\left( \cdot
\right) }\left( E\right) $.

The function $w\left( x,t\right) $ satisfies (\ref{43}), (\ref{2*})-(\ref{3*}%
) as well as the following conditions:%
\begin{equation*}
\int\limits_{r}^{diam\left( E\right) }\frac{w^{\frac{1}{p\left( x\right) }%
}(x,t)}{t}dt\lesssim w^{\frac{1}{p\left( x\right) }}\left( x,r\right) ,
\end{equation*}%
\begin{equation}
\int\limits_{r}^{diam\left( E\right) }\frac{w^{\frac{1}{p\left( x\right) }%
}(x,t)}{t^{1-\alpha \left( x\right) }}dt\lesssim r^{-\frac{\alpha \left(
x\right) p\left( x\right) }{q\left( x\right) -p\left( x\right) }},
\label{16}
\end{equation}%
where $p\left( x\right) <q\left( x\right) $. Then the operators $I_{\Omega
,\alpha \left( \cdot \right) }$ and $M_{\Omega ,\alpha \left( \cdot \right)
} $ are $\left( VL_{\Pi }^{p\left( \cdot \right) ,w^{\frac{1}{p\left( \cdot
\right) }}}\left( E\right) \rightarrow VL_{\Pi }^{q\left( \cdot \right) ,w^{%
\frac{1}{q\left( \cdot \right) }}}\left( E\right) \right) $-bounded.
Moreover,%
\begin{equation*}
\left\Vert I_{\Omega ,\alpha \left( \cdot \right) }f\right\Vert _{VL_{\Pi
}^{q\left( \cdot \right) ,w^{\frac{1}{q\left( \cdot \right) }}}\left(
E\right) }\lesssim \left\Vert f\right\Vert _{VL_{\Pi }^{p\left( \cdot
\right) ,w^{\frac{1}{p\left( \cdot \right) }}}\left( E\right) },
\end{equation*}%
\begin{equation*}
\left\Vert M_{\Omega ,\alpha \left( \cdot \right) }f\right\Vert _{VL_{\Pi
}^{q\left( \cdot \right) ,w^{\frac{1}{q\left( \cdot \right) }}}\left(
E\right) }\lesssim \left\Vert f\right\Vert _{VL_{\Pi }^{p\left( \cdot
\right) ,w^{\frac{1}{p\left( \cdot \right) }}}\left( E\right) }.
\end{equation*}
\end{theorem}

\begin{proof}
As in the proof of Theorem \ref{teo7}, we represent the function $f$ in the
form (\ref{42}) and have%
\begin{equation*}
I_{\Omega ,\alpha \left( \cdot \right) }f\left( x\right) =I_{\Omega ,\alpha
\left( \cdot \right) }f_{1}\left( x\right) +I_{\Omega ,\alpha \left( \cdot
\right) }f_{2}\left( x\right) .
\end{equation*}%
For $I_{\Omega ,\alpha \left( \cdot \right) }f_{1}\left( x\right) $, similar
to the proof of (\ref{12}), we obtain the following pointwise estimate:%
\begin{equation}
\left\vert I_{\Omega ,\alpha \left( \cdot \right) }f_{1}\left( x\right)
\right\vert \lesssim t^{\alpha \left( x\right) }M_{\Omega }f\left( x\right) .
\label{8}
\end{equation}%
For $I_{\Omega ,\alpha \left( \cdot \right) }f_{2}\left( x\right) $, similar
to the proof of (\ref{13*}), applying Fubini's theorem, H\"{o}lder's
inequality and (\ref{0*}), we get%
\begin{equation}
\left\vert I_{\Omega ,\alpha \left( \cdot \right) }f_{2}\left( x\right)
\right\vert \lesssim \int\limits_{r}^{diam\left( E\right) }t^{\alpha \left(
x\right) -\frac{n}{p\left( x\right) }-1}\left\Vert f\right\Vert
_{L_{p}\left( \tilde{B}(x,t)\right) }dt  \label{9}
\end{equation}%
and by (\ref{8}) and (\ref{9}) complete the proof of (\ref{15}).

Since $M_{\Omega ,\alpha \left( \cdot \right) }\left( f\right) \lesssim
I_{\left\vert \Omega \right\vert ,\alpha \left( \cdot \right) }\left(
\left\vert f\right\vert \right) $ (see Lemma \ref{lemma100}), it suffices to
treat only the case of the operator $I_{\Omega ,\alpha \left( \cdot \right)
} $. In this sense, by (\ref{15}) and (\ref{16}), we obtain%
\begin{equation*}
\left\vert I_{\Omega ,\alpha \left( \cdot \right) }f\left( x\right)
\right\vert \lesssim r^{\alpha \left( x\right) }M_{\Omega }f\left( x\right)
+r^{-\frac{\alpha \left( x\right) p\left( x\right) }{q\left( x\right)
-p\left( x\right) }}\left\Vert f\right\Vert _{VL_{\Pi }^{p\left( \cdot
\right) ,w\left( \cdot \right) }\left( E\right) }.
\end{equation*}%
Then, choosing $r=\left( \frac{\left\Vert f\right\Vert _{VL_{\Pi }^{p\left(
\cdot \right) ,w\left( \cdot \right) }\left( E\right) }}{M_{\Omega }f\left(
x\right) }\right) ^{\frac{q\left( x\right) -p\left( x\right) }{\alpha \left(
x\right) p\left( x\right) }}$ for every $x\in E$ supposing that $f$ is not
equal $0$, thus we have%
\begin{equation}
\left\vert I_{\Omega ,\alpha \left( \cdot \right) }f\left( x\right)
\right\vert \lesssim \left( M_{\Omega }f\left( x\right) \right) ^{\frac{%
p\left( x\right) }{q\left( x\right) }}\left\Vert f\right\Vert _{VL_{\Pi
}^{p\left( \cdot \right) ,w\left( \cdot \right) }\left( E\right) }^{1-\frac{%
p\left( x\right) }{q\left( x\right) }}.  \label{17}
\end{equation}%
Finally, by Definition \ref{definition}, (\ref{17}) and (\ref{20}) we get%
\begin{eqnarray*}
\left\Vert I_{\Omega ,\alpha \left( \cdot \right) }f\right\Vert _{VL_{\Pi
}^{q\left( \cdot \right) ,w^{\frac{1}{q\left( \cdot \right) }}}\left(
E\right) } &=&\sup\limits_{x\in \Pi ,r>0}\frac{r^{-\frac{n}{q\left( x\right) 
}}\left\Vert I_{\Omega ,\alpha \left( \cdot \right) }f\right\Vert
_{L^{p\left( \cdot \right) }\left( \tilde{B}(x,r)\right) }}{w(x,r)^{\frac{1}{%
q\left( x\right) }}} \\
&\lesssim &\left\Vert f\right\Vert _{VL_{\Pi }^{p\left( \cdot \right)
,w\left( \cdot \right) }\left( E\right) }^{1-\frac{p\left( x\right) }{%
q\left( x\right) }}\sup\limits_{x\in \Pi ,r>0}\frac{r^{-\frac{n}{q\left(
x\right) }}}{w(x,r)^{\frac{1}{q\left( x\right) }}}\left\Vert M_{\Omega
}f\right\Vert _{L^{p\left( \cdot \right) }\left( \tilde{B}(x,r)\right) }^{%
\frac{p\left( x\right) }{q\left( x\right) }} \\
&\lesssim &\left\Vert f\right\Vert _{VL_{\Pi }^{p\left( \cdot \right)
,w\left( \cdot \right) }\left( E\right) }^{1-\frac{p\left( x\right) }{%
q\left( x\right) }}\left( \sup\limits_{x\in \Pi ,r>0}\frac{r^{-\frac{n}{%
p\left( x\right) }}}{w(x,r)^{\frac{1}{p\left( x\right) }}}\left\Vert
M_{\Omega }f\right\Vert _{L^{p\left( \cdot \right) }\left( \tilde{B}%
(x,r)\right) }\right) ^{\frac{p\left( x\right) }{q\left( x\right) }} \\
&\lesssim &\left\Vert f\right\Vert _{VL_{\Pi }^{p\left( \cdot \right)
,w\left( \cdot \right) }\left( E\right) }^{1-\frac{p\left( x\right) }{%
q\left( x\right) }}\left\Vert M_{\Omega }f\right\Vert _{VL_{\Pi }^{p\left(
\cdot \right) ,w^{\frac{1}{p\left( \cdot \right) }}}\left( E\right) }^{\frac{%
p\left( x\right) }{q\left( x\right) }} \\
&\lesssim &\left\Vert f\right\Vert _{VL_{\Pi }^{p\left( \cdot \right) ,w^{%
\frac{1}{p\left( \cdot \right) }}}\left( E\right) }
\end{eqnarray*}%
if $p\left( x\right) <q\left( x\right) $ and 
\begin{equation*}
\lim_{r\rightarrow 0}\sup\limits_{x\in \Pi }\frac{r^{-\frac{n}{q\left(
x\right) }}\left\Vert I_{\Omega ,\alpha \left( \cdot \right) }f\right\Vert
_{L^{q\left( \cdot \right) }\left( \tilde{B}(x,r)\right) }}{w_{2}(x,t)^{%
\frac{1}{q\left( x\right) }}}\lesssim \lim_{r\rightarrow 0}\sup\limits_{x\in
\Pi }\frac{r^{-\frac{n}{p\left( x\right) }}\left\Vert f\right\Vert
_{L^{p\left( \cdot \right) }\left( \tilde{B}(x,r)\right) }}{w_{1}(x,t)^{%
\frac{1}{p\left( x\right) }}}=0,
\end{equation*}%
which completes the proof of Theorem \ref{teo8}.
\end{proof}

\begin{corollary}
Obviously, under the conditions of Theorem \ref{teo8}, if the rough $\left(
p,q\right) $-admissible $T_{\Omega ,\alpha \left( \cdot \right) }$-potential
type operator is $\left( L^{p\left( \cdot \right) }\left( E\right)
\rightarrow L^{q\left( \cdot \right) }\left( E\right) \right) $-bounded and
satisfies (\ref{e1}), the result in Theorem \ref{teo8} still holds.
\end{corollary}

\begin{remark}
Let $E$ be a bounded open set and $\lambda \left( x\right) $ be a measurable
function on $E$ with values in $\left[ 0,n\right] $. Then, the variable
exponent vanishing Morrey space $VL_{\Pi }^{p\left( \cdot \right) ,\lambda
\left( \cdot \right) }\equiv VL_{\Pi }^{p\left( \cdot \right) ,\lambda
\left( \cdot \right) }\left( E\right) $ is defined by%
\begin{equation*}
VL_{\Pi }^{p\left( \cdot \right) ,\lambda \left( \cdot \right) }\equiv
VL_{\Pi }^{p\left( \cdot \right) ,\lambda \left( \cdot \right) }\left(
E\right) =\left\{ 
\begin{array}{c}
f\in L^{p\left( \cdot \right) ,\lambda \left( \cdot \right) }\left( E\right)
: \\ 
\Vert f\Vert _{VL_{\Pi }^{p\left( \cdot \right) ,\lambda \left( \cdot
\right) }}=\lim\limits_{r\rightarrow 0}\sup\limits_{\substack{ x\in E  \\ %
0<t<r}}t^{-\frac{\lambda \left( x\right) }{p\left( x\right) }}\left\Vert
f\chi _{\tilde{B}(x,t)}\right\Vert _{L^{p\left( \cdot \right) }\left(
E\right) }=0%
\end{array}%
\right\} .
\end{equation*}
\end{remark}

\begin{corollary}
\label{Corollary1}Let $E$, $\Omega $, $p\left( x\right) $, $\alpha \left(
x\right) $ be the same as in Theorem \ref{teo7}. Define $q\left( x\right) $
by $\frac{1}{q\left( x\right) }=\frac{1}{p\left( x\right) }-\frac{\alpha
\left( x\right) }{n-\lambda \left( x\right) }$. Let also the following
conditions hold:%
\begin{equation*}
\lambda \left( x\right) \geq 0,\qquad \limfunc{esssup}\limits_{x\in E}\left[
\lambda \left( x\right) +\alpha \left( x\right) p\left( x\right) \right] <n.
\end{equation*}%
Then for $\left( p_{-}\right) ^{\prime }\leq s$, the operators $I_{\Omega
,\alpha \left( \cdot \right) }$ and $M_{\Omega ,\alpha \left( \cdot \right)
} $ are $\left( VL_{\Pi }^{p\left( \cdot \right) ,\lambda \left( \cdot
\right) }\left( E\right) \rightarrow VL_{\Pi }^{q\left( \cdot \right)
,\lambda \left( \cdot \right) }\left( E\right) \right) $-bounded. Moreover,%
\begin{equation*}
\left\Vert I_{\Omega ,\alpha \left( \cdot \right) }f\right\Vert _{VL_{\Pi
}^{q\left( \cdot \right) ,\lambda \left( \cdot \right) }\left( E\right)
}\lesssim \left\Vert f\right\Vert _{VL_{\Pi }^{p\left( \cdot \right)
,\lambda \left( \cdot \right) }\left( E\right) },
\end{equation*}%
\begin{equation*}
\left\Vert M_{\Omega ,\alpha \left( \cdot \right) }f\right\Vert _{VL_{\Pi
}^{q\left( \cdot \right) ,\lambda \left( \cdot \right) }\left( E\right)
}\lesssim \left\Vert f\right\Vert _{VL_{\Pi }^{p\left( \cdot \right)
,\lambda \left( \cdot \right) }\left( E\right) }.
\end{equation*}
\end{corollary}

In the case of $\lambda \left( x\right) \equiv 0$, for the spaces $%
L^{p\left( \cdot \right) }\left( E\right) $, from Corollary \ref{Corollary1}
we get the following:

\begin{corollary}
Let $E$, $\Omega $, $p\left( x\right) $, $q\left( x\right) $, $\alpha \left(
x\right) $ be the same as in Theorem \ref{teo7}. Then, the operators $%
I_{\Omega ,\alpha \left( \cdot \right) }$ and $M_{\Omega ,\alpha \left(
\cdot \right) }$ are $\left( L^{p\left( \cdot \right) }\left( E\right)
\rightarrow L^{q\left( \cdot \right) }\left( E\right) \right) $-bounded.
Moreover,%
\begin{equation*}
\left\Vert I_{\Omega ,\alpha \left( \cdot \right) }f\right\Vert _{L^{q\left(
\cdot \right) }\left( E\right) }\lesssim \left\Vert f\right\Vert
_{L^{p\left( \cdot \right) }\left( E\right) },
\end{equation*}%
\begin{equation*}
\left\Vert M_{\Omega ,\alpha \left( \cdot \right) }f\right\Vert _{L^{q\left(
\cdot \right) }\left( E\right) }\lesssim \left\Vert f\right\Vert
_{L^{p\left( \cdot \right) }\left( E\right) }.
\end{equation*}
\end{corollary}

\subsection{Preliminaries on \textbf{variable exponent generalized Campanato
spaces }$\mathcal{C}_{\Pi }^{q\left( \cdot \right) ,\protect\gamma \left(
\cdot \right) }$}

\begin{flushleft}
In this section, we first introduce the variable exponent generalized
Campanato spaces and then obtain the boundedness of the commutators of the
operators $I_{\Omega ,\alpha \left( \cdot \right) }$, $M_{\Omega ,\alpha
\left( \cdot \right) }$, $T_{\Omega }$ and $M_{\Omega }$ on the spaces $%
VL_{\Pi }^{p\left( \cdot \right) ,w\left( \cdot \right) }\left( E\right) $.
\end{flushleft}

\begin{definition}
Let $1\leq q\left( \cdot \right) \leq q^{+}<\infty $ and $0\leq \gamma
\left( \cdot \right) <\frac{1}{n}$. Define the generalized Campanato space $%
\mathcal{C}_{\Pi }^{q\left( \cdot \right) ,\gamma \left( \cdot \right) }$
with variable exponents $q\left( \cdot \right) $, $\gamma \left( \cdot
\right) $ as follows: 
\begin{equation*}
\mathcal{C}_{\Pi }^{q\left( \cdot \right) ,\gamma \left( \cdot \right)
}\left( E\right) =\left\{ f\in L_{loc}^{q\left( \cdot \right) }\left( \tilde{%
B}\left( x,r\right) \right) :\left\Vert f\right\Vert _{\mathcal{C}_{\Pi
}^{q\left( \cdot \right) ,\gamma \left( \cdot \right) }\left( E\right)
}<\infty \right\} ,
\end{equation*}%
where 
\begin{equation*}
\left\Vert f\right\Vert _{\mathcal{C}_{\Pi }^{q\left( \cdot \right) ,\gamma
\left( \cdot \right) }\left( E\right) }=\sup_{x\in \Pi ,r>0}\left\vert
B(x,r)\right\vert ^{-\frac{1}{q(x)}-\gamma (x)}\left\Vert f-f_{B\left(
x,r\right) }\right\Vert _{L^{q\left( \cdot \right) }\left( \tilde{B}\left(
x,r\right) \right) }
\end{equation*}%
such that%
\begin{equation}
\left\Vert f-f_{B\left( x,r\right) }\right\Vert _{L^{q\left( \cdot \right)
}\left( \tilde{B}\left( x,r\right) \right) }\lesssim r^{\frac{n}{q\left(
x\right) }+n\gamma (x)}\left\Vert f\right\Vert _{\mathcal{C}_{\Pi }^{q\left(
\cdot \right) ,\gamma \left( \cdot \right) }\left( E\right) }.  \label{14}
\end{equation}%
When $\Pi =\left\{ x_{0}\right\} $ and $\Pi =E$, $\mathcal{C}_{\Pi
}^{q\left( \cdot \right) ,\gamma \left( \cdot \right) }\left( E\right) $
turns into the local generalized Campanato space $\mathcal{C}_{\left\{
x_{0}\right\} }^{q\left( \cdot \right) ,\gamma \left( \cdot \right) }\left(
E\right) $ and the global generalized Campanato space $\mathcal{C}%
_{E}^{q\left( \cdot \right) ,\gamma \left( \cdot \right) }\left( E\right) $,
respectively. If $q\left( \cdot \right) $, $\gamma \left( \cdot \right) $
are constant functions and $\Pi =E$, then the variable exponent generalized
Campanato space $\mathcal{C}_{\Pi }^{q\left( \cdot \right) ,\gamma \left(
\cdot \right) }\left( E\right) $ is exactly the usual Campanato space $%
\mathcal{C}^{q,\gamma }\left( E\right) $. If $\gamma \left( \cdot \right)
\equiv 0$ and $q\left( \cdot \right) \equiv q$, the generalized Campanato
space $\mathcal{C}_{\Pi }^{q\left( \cdot \right) ,\gamma \left( \cdot
\right) }\left( E\right) $ is just the central $BMO\left( E\right) $(the
local version of $BMO\left( E\right) $).
\end{definition}

\begin{theorem}
\label{teo9}Let $E$, $\Omega $, $p\left( x\right) $, $q\left( x\right) $, $%
\alpha \left( x\right) $ be the same as in Theorem \ref{teo7}. Let also $%
\frac{1}{p\left( \cdot \right) }=\frac{1}{p_{1}\left( \cdot \right) }+\frac{1%
}{p_{2}\left( \cdot \right) }$, $\frac{1}{q_{1}\left( \cdot \right) }=\frac{1%
}{p_{1}\left( \cdot \right) }-\frac{\alpha \left( \cdot \right) }{n}$ and $%
b\in \mathcal{C}_{\Pi }^{p_{2}\left( \cdot \right) ,\gamma \left( \cdot
\right) }\left( E\right) $. Suppose that $p_{1}\left( \cdot \right) $, $%
p_{2}\left( \cdot \right) $, $q\left( \cdot \right) $, $q_{1}\left( \cdot
\right) $ and $\alpha \left( \cdot \right) $ satisfy (\ref{1}). Then, for $%
\frac{s}{s-1}<p^{-}\leq p\left( \cdot \right) <\frac{n}{\alpha \left( \cdot
\right) }$ the following pointwise estimate 
\begin{equation}
\Vert \left[ b,I_{\Omega ,\alpha \left( \cdot \right) }\right] f\Vert
_{L_{q}(\tilde{B}(x,r))}\lesssim \Vert b\Vert _{\mathcal{C}_{\Pi
}^{p_{2}\left( \cdot \right) ,\gamma \left( \cdot \right) }}r^{\frac{n}{%
q\left( x\right) }}\dint\limits_{2r}^{diam\left( E\right) }\left( 1+\ln 
\frac{t}{r}\right) t^{n\gamma \left( x\right) -\frac{n}{q_{1}\left( x\right) 
}-1}\left\Vert f\right\Vert _{L^{p_{1}\left( \cdot \right) }\left( \tilde{B}%
(x,t)\right) }dt  \label{45}
\end{equation}%
holds for any ball $\tilde{B}(x,r)$ and for all $f\in L_{loc}^{p_{1}\left(
\cdot \right) }\left( E\right) $.

If the functions $w_{1}\left( x,r\right) $ and $w_{2}\left( x,r\right) $
satisfy (\ref{43}) as well as the following Zygmund condition%
\begin{equation}
\int\limits_{r}^{diam\left( E\right) }\left( 1+\ln \frac{t}{r}\right) \frac{%
w_{1}^{\frac{1}{p_{1}\left( x\right) }}(x,t)}{t^{1-\left( \alpha \left(
x\right) +n\gamma \left( x\right) \right) }}dt\lesssim w_{2}^{\frac{1}{%
q\left( x\right) }}(x,r)\,,\qquad r\in \left( 0,diam\left( E\right) \right]
\label{46}
\end{equation}%
and additionally these functions satisfy the conditions (\ref{2*})-(\ref{3*}%
), 
\begin{equation*}
d_{\delta }:=\dint\limits_{\delta }^{diam\left( E\right) }\sup_{x\in \Pi
}\left( 1+\ln \frac{t}{r}\right) \frac{w_{1}^{\frac{1}{p_{1}\left( x\right) }%
}(x,t)}{t^{1-\left( \alpha \left( x\right) +n\gamma \left( x\right) \right) }%
}dt<\infty ,\qquad \delta >0,
\end{equation*}%
then the operators $\left[ b,I_{\Omega ,\alpha \left( \cdot \right) }\right] 
$ and $\left[ b,M_{\Omega ,\alpha \left( \cdot \right) }\right] $ are $%
\left( VL_{\Pi }^{p_{1}\left( \cdot \right) ,w_{1}\left( \cdot \right)
}\left( E\right) \rightarrow VL_{\Pi }^{q\left( \cdot \right) ,w_{2}\left(
\cdot \right) }\left( E\right) \right) $-bounded. Moreover,%
\begin{equation*}
\left\Vert \left[ b,I_{\Omega ,\alpha \left( \cdot \right) }\right]
f\right\Vert _{VL_{\Pi }^{q\left( \cdot \right) ,w_{2}\left( \cdot \right)
}\left( E\right) }\lesssim \Vert b\Vert _{\mathcal{C}_{\Pi }^{p_{2}\left(
\cdot \right) ,\gamma \left( \cdot \right) }}\left\Vert f\right\Vert
_{VL_{\Pi }^{p_{1}\left( \cdot \right) ,w_{1}\left( \cdot \right) }\left(
E\right) },
\end{equation*}%
\begin{equation*}
\left\Vert \left[ b,M_{\Omega ,\alpha \left( \cdot \right) }\right]
f\right\Vert _{VL_{\Pi }^{q\left( \cdot \right) ,w_{2}\left( \cdot \right)
}\left( E\right) }\lesssim \Vert b\Vert _{\mathcal{C}_{\Pi }^{p_{2}\left(
\cdot \right) ,\gamma \left( \cdot \right) }}\left\Vert f\right\Vert
_{VL_{\Pi }^{p_{1}\left( \cdot \right) ,w_{1}\left( \cdot \right) }\left(
E\right) }.
\end{equation*}
\end{theorem}

\begin{proof}
Since $\left[ b,M_{\Omega ,\alpha \left( \cdot \right) }\right] \left(
f\right) \lesssim \left[ b,I_{\left\vert \Omega \right\vert ,\alpha \left(
\cdot \right) }\right] \left( \left\vert f\right\vert \right) $, it suffices
to treat only the case of the operator $\left[ b,I_{\Omega ,\alpha \left(
\cdot \right) }\right] $. As in the proof of Theorem \ref{teo7}, we
represent the function $f$ in the form (\ref{42}) and have%
\begin{eqnarray*}
\left[ b,I_{\Omega ,\alpha \left( \cdot \right) }\right] f\left( x\right)
&=&\left( b\left( x\right) -b_{B\left( x,r\right) }\right) I_{\Omega ,\alpha
\left( \cdot \right) }f_{1}\left( x\right) -I_{\Omega ,\alpha \left( \cdot
\right) }\left( \left( b\left( \cdot \right) -b_{B\left( x,r\right) }\right)
f_{1}\right) \left( x\right) \\
&&+\left( b\left( x\right) -b_{B\left( x,r\right) }\right) I_{\Omega ,\alpha
\left( \cdot \right) }f_{2}\left( x\right) -I_{\Omega ,\alpha \left( \cdot
\right) }\left( \left( b\left( \cdot \right) -b_{B\left( x,r\right) }\right)
f_{2}\right) \left( x\right) \\
&\equiv &F_{1}+F_{2}+F_{3}+F_{4}.
\end{eqnarray*}%
Hence we get%
\begin{equation*}
\Vert \left[ b,I_{\Omega ,\alpha \left( \cdot \right) }\right] f\Vert
_{L^{q\left( \cdot \right) }(\tilde{B}(x,r))}\leq \left\Vert
F_{1}\right\Vert _{L^{q\left( \cdot \right) }\left( \tilde{B}(x,r)\right)
}+\left\Vert F_{2}\right\Vert _{L^{q\left( \cdot \right) }\left( \tilde{B}%
(x,r)\right) }+\left\Vert F_{3}\right\Vert _{L^{q\left( \cdot \right)
}\left( \tilde{B}(x,r)\right) }+\left\Vert F_{4}\right\Vert _{L^{q\left(
\cdot \right) }\left( \tilde{B}(x,r)\right) }.
\end{equation*}%
First, we use the H\"{o}lder's inequality such that $\frac{1}{q\left( \cdot
\right) }=\frac{1}{p_{2}\left( \cdot \right) }+\frac{1}{q_{1}\left( \cdot
\right) }$, the boundedness of $I_{\Omega ,\alpha \left( \cdot \right) }$
from $L^{p\left( \cdot \right) }$ into $L^{q\left( \cdot \right) }$ (see
Theorem \ref{teo1}) and (\ref{14}) to estimate $\left\Vert F_{1}\right\Vert
_{L^{q\left( \cdot \right) }\left( \tilde{B}(x,r)\right) }$, and we obtain%
\begin{eqnarray*}
\left\Vert F_{1}\right\Vert _{L^{q\left( \cdot \right) }\left( \tilde{B}%
(x,r)\right) } &=&\left\Vert \left( b\left( \cdot \right) -b_{B}\right)
I_{\Omega ,\alpha \left( \cdot \right) }f_{1}\left( \cdot \right)
\right\Vert _{L^{q\left( \cdot \right) }\left( \tilde{B}(x,r)\right) } \\
&\lesssim &\left\Vert \left( b\left( \cdot \right) -b_{B}\right) \right\Vert
_{L^{p_{2}\left( \cdot \right) }\left( \tilde{B}(x,r)\right) }\left\Vert
I_{\Omega ,\alpha \left( \cdot \right) }f_{1}\left( \cdot \right)
\right\Vert _{L^{q_{1}\left( \cdot \right) }\left( \tilde{B}(x,r)\right) } \\
&\lesssim &r^{\frac{n}{p_{2}\left( x\right) }+n\gamma (x)}\left\Vert
b\right\Vert _{\mathcal{C}_{\Pi }^{p_{2}\left( \cdot \right) ,\gamma \left(
\cdot \right) }}\left\Vert f_{1}\right\Vert _{L^{_{p_{1}}\left( \cdot
\right) }\left( \tilde{B}(x,r)\right) } \\
&=&\left\Vert b\right\Vert _{\mathcal{C}_{\Pi }^{p_{2}\left( \cdot \right)
,\gamma \left( \cdot \right) }}r^{\frac{n}{p_{2}\left( x\right) }+\frac{n}{%
q_{1}\left( x\right) }+n\gamma \left( x\right) }\left\Vert f\right\Vert
_{L^{_{p_{1}}\left( \cdot \right) }\left( \tilde{B}\left( x,2r\right)
\right) }\int\limits_{2r}^{diam\left( E\right) }t^{-1-\frac{n}{q_{1}\left(
x\right) }}dt \\
&\lesssim &\left\Vert b\right\Vert _{\mathcal{C}_{\Pi }^{p_{2}\left( \cdot
\right) ,\gamma \left( \cdot \right) }}r^{\frac{n}{q\left( x\right) }%
}\int\limits_{2r}^{diam\left( E\right) }\left( 1+\ln \frac{t}{r}\right)
\left\Vert f\right\Vert _{L^{_{p_{1}}\left( \cdot \right) }\left( \tilde{B}%
\left( x,t\right) \right) }t^{n\gamma \left( x\right) -\frac{n}{q_{1}\left(
x\right) }-1}dt.
\end{eqnarray*}%
Second, for $\left\Vert F_{2}\right\Vert _{L^{q\left( \cdot \right) }\left( 
\tilde{B}(x,r)\right) }$, applying the boundedness of $I_{\Omega ,\alpha
\left( \cdot \right) }$ from $L^{p\left( \cdot \right) }$ into $L^{q\left(
\cdot \right) }$ (see Theorem \ref{teo1}), generalized H\"{o}lder's
inequality such that $\frac{1}{p\left( \cdot \right) }=\frac{1}{p_{1}\left(
\cdot \right) }+\frac{1}{p_{2}\left( \cdot \right) }$, $\frac{1}{q\left(
\cdot \right) }=\frac{1}{p_{2}\left( \cdot \right) }+\frac{1}{q_{1}\left(
\cdot \right) }$ and (\ref{14}), we know that%
\begin{eqnarray*}
\left\Vert F_{2}\right\Vert _{L^{q\left( \cdot \right) }\left( \tilde{B}%
(x,r)\right) } &=&\left\Vert I_{\Omega ,\alpha \left( \cdot \right) }\left(
b\left( \cdot \right) -b_{B\left( x,r\right) }\right) f_{1}\right\Vert
_{L^{q\left( \cdot \right) }\left( \tilde{B}(x,r)\right) } \\
&\lesssim &\left\Vert \left( b\left( \cdot \right) -b_{B}\right)
f_{1}\right\Vert _{L^{p\left( \cdot \right) }\left( \tilde{B}(x,r)\right) }
\\
&\lesssim &\left\Vert \left( b\left( \cdot \right) -b_{B}\right) \right\Vert
_{L^{p_{2}\left( \cdot \right) }\left( \tilde{B}(x,r)\right) }\left\Vert
f_{1}\right\Vert _{L^{p_{1}\left( \cdot \right) }\left( \tilde{B}%
(x,r)\right) } \\
&\lesssim &\left\Vert f\right\Vert _{\mathcal{C}_{\Pi }^{p_{2}\left( \cdot
\right) ,\gamma \left( \cdot \right) }}r^{\frac{n}{p_{2}\left( x\right) }+%
\frac{n}{q_{1}\left( x\right) }+n\gamma \left( x\right) }\left\Vert
f\right\Vert _{L^{_{p_{1}}\left( \cdot \right) }\left( \tilde{B}\left(
x,2r\right) \right) }\int\limits_{2r}^{diam\left( E\right) }t^{-1-\frac{n}{%
q_{1}\left( x\right) }}dt \\
&\lesssim &\left\Vert b\right\Vert _{\mathcal{C}_{\Pi }^{p_{2}\left( \cdot
\right) ,\gamma \left( \cdot \right) }}r^{\frac{n}{q\left( x\right) }%
}\int\limits_{2r}^{diam\left( E\right) }\left( 1+\ln \frac{t}{r}\right)
\left\Vert f\right\Vert _{L^{_{p_{1}}\left( \cdot \right) }\left( \tilde{B}%
\left( x,t\right) \right) }t^{n\gamma \left( x\right) -\frac{n}{q_{1}\left(
x\right) }-1}dt.
\end{eqnarray*}%
Third, for $\left\Vert F_{3}\right\Vert _{L^{q\left( \cdot \right) }\left( 
\tilde{B}(x,r)\right) }$, similar to the proof of (\ref{13*}), when $\frac{s%
}{s-1}\leq p_{1}\left( \cdot \right) $, by Fubini's theorem, generalized H%
\"{o}lder's inequality and (\ref{0*}), we have%
\begin{eqnarray}
\left\vert I_{\Omega ,\alpha \left( \cdot \right) }f_{2}\left( x\right)
\right\vert &\lesssim &\dint\limits_{E\setminus \tilde{B}\left( x,2r\right) }%
\frac{\left\vert \Omega (z-y)\right\vert \left\vert f(y)\right\vert }{%
|x-y|^{n-\alpha \left( x\right) }}dy  \notag \\
&\lesssim &\dint\limits_{2r}^{diam\left( E\right) }\left\Vert f\right\Vert
_{L^{p_{1}\left( \cdot \right) }\left( \tilde{B}\left( x,t\right) \right)
}t^{-1-\frac{n}{q_{1}\left( x\right) }}dt.  \label{21}
\end{eqnarray}%
Thus, by generalized H\"{o}lder's inequality such that $\frac{1}{q\left(
\cdot \right) }=\frac{1}{p_{2}\left( \cdot \right) }+\frac{1}{q_{1}\left(
\cdot \right) }$, (\ref{14}) and (\ref{21}), we obtain%
\begin{eqnarray*}
\left\Vert F_{3}\right\Vert _{L^{q\left( \cdot \right) }\left( \tilde{B}%
(x,r)\right) } &=&\left\Vert \left( b\left( \cdot \right) -b_{B\left(
x,r\right) }\right) I_{\Omega ,\alpha \left( \cdot \right) }f_{2}\left(
\cdot \right) \right\Vert _{L^{q\left( \cdot \right) }\left( \tilde{B}%
(x,r)\right) } \\
&\lesssim &\left\Vert \left( b\left( \cdot \right) -b_{B\left( x,r\right)
}\right) \right\Vert _{L^{p_{2}\left( \cdot \right) }\left( \tilde{B}%
(x,r)\right) }\left\Vert I_{\Omega ,\alpha \left( \cdot \right) }f_{2}\left(
\cdot \right) \right\Vert _{L^{q_{1}\left( \cdot \right) }\left( \tilde{B}%
(x,r)\right) } \\
&\lesssim &r^{\frac{n}{p_{2}\left( x\right) }+n\gamma (x)}\left\Vert
b\right\Vert _{\mathcal{C}_{\Pi }^{p_{2}\left( \cdot \right) ,\gamma \left(
\cdot \right) }}r^{\frac{n}{q_{1}\left( x\right) }}\int\limits_{2r}^{diam%
\left( E\right) }\left\Vert f\right\Vert _{L^{_{p_{1}}\left( \cdot \right)
}\left( \tilde{B}\left( x,t\right) \right) }t^{-1-\frac{n}{q_{1}\left(
x\right) }}dt \\
&\lesssim &\left\Vert b\right\Vert _{\mathcal{C}_{\Pi }^{p_{2}\left( \cdot
\right) ,\gamma \left( \cdot \right) }}r^{\frac{n}{q\left( x\right) }%
}\int\limits_{2r}^{diam\left( E\right) }\left( 1+\ln \frac{t}{r}\right)
\left\Vert f\right\Vert _{L^{_{p_{1}}\left( \cdot \right) }\left( \tilde{B}%
\left( x,t\right) \right) }t^{n\gamma \left( x\right) -\frac{n}{q_{1}\left(
x\right) }-1}dt.
\end{eqnarray*}

Finally, we consider the term $\left\Vert F_{4}\right\Vert _{L^{q\left(
\cdot \right) }\left( \tilde{B}(x,r)\right) }=\left\Vert I_{\Omega ,\alpha
\left( \cdot \right) }\left( \left( b\left( \cdot \right) -b_{B\left(
x,r\right) }\right) f_{2}\right) \left( \cdot \right) \right\Vert
_{L^{q\left( \cdot \right) }\left( \tilde{B}(x,r)\right) }$. For $z\in
B\left( x,r\right) $, when $\frac{s}{s-1}\leq p\left( \cdot \right) $, by
the Fubini's theorem, applying the generalized H\"{o}lder's inequality and
from (\ref{0*}) and (\ref{14}) we have

$\left\vert I_{\Omega ,\alpha \left( \cdot \right) }\left( \left( b\left(
\cdot \right) -b_{B\left( x,r\right) }\right) f_{2}\right) \left( z\right)
\right\vert $

$\lesssim \dint\limits_{2r}^{diam\left( E\right) }\left\vert b\left(
y\right) -b_{B\left( x,r\right) }\right\vert \left\vert \Omega \left(
z-y\right) \right\vert \frac{\left\vert f\left( y\right) \right\vert }{%
\left\vert x-y\right\vert ^{n-\alpha \left( x\right) }}dy$

$\approx \dint\limits_{2r}^{diam\left( E\right) }\dint\limits_{2r<\left\vert
x-y\right\vert <t}\left\vert b\left( y\right) -b_{B\left( x,r\right)
}\right\vert \left\vert \Omega \left( z-y\right) \right\vert \left\vert
f\left( y\right) \right\vert dy\frac{dt}{t^{n-\alpha \left( x\right) +1}}$

$\lesssim \dint\limits_{2r}^{diam\left( E\right) }\dint\limits_{B\left(
x,t\right) }\left\vert b\left( y\right) -b_{B\left( x,t\right) }\right\vert
\left\vert \Omega \left( z-y\right) \right\vert \left\vert f\left( y\right)
\right\vert dy\frac{dt}{t^{n-\alpha \left( x\right) +1}}$

$+\dint\limits_{2r}^{diam\left( E\right) }\dint\limits_{B\left( x,t\right)
}\left\vert b_{B\left( x,r\right) }-b_{B\left( x,t\right) }\right\vert
\left\vert \Omega \left( z-y\right) \right\vert \left\vert f\left( y\right)
\right\vert dy\frac{dt}{t^{n-\alpha \left( x\right) +1}}$

$\lesssim \dint\limits_{2r}^{diam\left( E\right) }\left\Vert \left( b\left(
\cdot \right) -b_{B\left( x,r\right) }\right) \right\Vert _{L^{p_{2}\left(
\cdot \right) }\left( \tilde{B}\left( x,t\right) \right) }\left\Vert
f\right\Vert _{L^{p_{1}\left( \cdot \right) }\left( \tilde{B}\left(
x,t\right) \right) }\left\Vert \Omega \left( z-\cdot \right) \right\Vert
_{L_{s}\left( \tilde{B}\left( x,t\right) \right) }t^{-\frac{n}{q\left(
x\right) }-\frac{n}{s}-1}dt$

$+\dint\limits_{2r}^{diam\left( E\right) }\left\Vert \left( b_{B\left(
x,r\right) }-b_{B\left( x,t\right) }\right) \right\Vert _{L^{p_{2}\left(
\cdot \right) }\left( \tilde{B}\left( x,t\right) \right) }\left\Vert
f\right\Vert _{L^{p_{1}\left( \cdot \right) }\left( \tilde{B}\left(
x,t\right) \right) }\left\Vert \Omega \left( z-\cdot \right) \right\Vert
_{L_{s}\left( \tilde{B}\left( x,t\right) \right) }t^{-\frac{n}{q\left(
x\right) }-\frac{n}{s}-1}dt$

$\lesssim \dint\limits_{2r}^{diam\left( E\right) }\left\Vert \left( b\left(
\cdot \right) -b_{B\left( x,r\right) }\right) \right\Vert _{L^{p_{2}\left(
\cdot \right) }\left( \tilde{B}\left( x,t\right) \right) }\left\Vert
f\right\Vert _{L^{p_{1}\left( \cdot \right) }\left( \tilde{B}\left(
x,t\right) \right) }t^{-\frac{n}{q\left( x\right) }-1}dt$

$+\left\Vert b\right\Vert _{\mathcal{C}_{\Pi }^{p_{2}\left( \cdot \right)
,\gamma \left( \cdot \right) }}\int\limits_{2r}^{diam\left( E\right) }\left(
1+\ln \frac{t}{r}\right) \left\Vert f\right\Vert _{L^{_{p_{1}}\left( \cdot
\right) }\left( \tilde{B}\left( x,t\right) \right) }t^{n\gamma \left(
x\right) -\frac{n}{q_{1}\left( x\right) }-1}dt$%
\begin{equation}
\lesssim \left\Vert b\right\Vert _{\mathcal{C}_{\Pi }^{p_{2}\left( \cdot
\right) ,\gamma \left( \cdot \right) }}\int\limits_{2r}^{diam\left( E\right)
}\left( 1+\ln \frac{t}{r}\right) \left\Vert f\right\Vert _{L^{_{p_{1}}\left(
\cdot \right) }\left( \tilde{B}\left( x,t\right) \right) }t^{n\gamma \left(
x\right) -\frac{n}{q_{1}\left( x\right) }-1}dt.  \label{19}
\end{equation}%
Then, by (\ref{19}) we have%
\begin{eqnarray*}
\left\Vert F_{4}\right\Vert _{L^{q\left( \cdot \right) }\left( \tilde{B}%
(x,r)\right) } &=&\left\Vert I_{\Omega ,\alpha \left( \cdot \right) }\left(
\left( b\left( \cdot \right) -b_{B\left( x,r\right) }\right) f_{2}\right)
\left( x\right) \right\Vert _{L^{q\left( \cdot \right) }\left( \tilde{B}%
(x,r)\right) } \\
&\lesssim &\left\Vert b\right\Vert _{\mathcal{C}_{\Pi }^{p_{2}\left( \cdot
\right) ,\gamma \left( \cdot \right) }}r^{\frac{n}{q\left( x\right) }%
}\int\limits_{2r}^{diam\left( E\right) }\left( 1+\ln \frac{t}{r}\right)
\left\Vert f\right\Vert _{L^{_{p_{1}}\left( \cdot \right) }\left( \tilde{B}%
\left( x,t\right) \right) }t^{n\gamma \left( x\right) -\frac{n}{q_{1}\left(
x\right) }-1}dt.
\end{eqnarray*}%
Combining all the estimates of $\left\Vert F_{1}\right\Vert _{L^{q\left(
\cdot \right) }\left( \tilde{B}(x,r)\right) }$, $\left\Vert F_{2}\right\Vert
_{L^{q\left( \cdot \right) }\left( \tilde{B}(x,r)\right) }$, $\left\Vert
F_{3}\right\Vert _{L^{q\left( \cdot \right) }\left( \tilde{B}(x,r)\right) }$%
, $\left\Vert F_{4}\right\Vert _{L^{q\left( \cdot \right) }\left( \tilde{B}%
(x,r)\right) }$, we get (\ref{45}).

At last, by Definition \ref{definition}, (\ref{45}) and (\ref{46}) we get%
\begin{eqnarray*}
\left\Vert \left[ b,I_{\Omega ,\alpha \left( \cdot \right) }\right]
f\right\Vert _{VL_{\Pi }^{q\left( \cdot \right) ,w_{2}\left( \cdot \right)
}\left( E\right) } &=&\sup\limits_{x\in \Pi ,r>0}\frac{r^{-\frac{n}{q\left(
x\right) }}\left\Vert \left[ b,I_{\Omega ,\alpha \left( \cdot \right) }%
\right] f\right\Vert _{L^{q\left( \cdot \right) }\left( \tilde{B}%
(x,r)\right) }}{w_{2}(x,r)^{\frac{1}{q\left( x\right) }}} \\
&\lesssim &\Vert b\Vert _{\mathcal{C}_{\Pi }^{p_{2}\left( \cdot \right)
,\gamma \left( \cdot \right) }}\sup\limits_{x\in \Pi ,r>0}\frac{1}{%
w_{2}(x,r)^{\frac{1}{q\left( x\right) }}}\dint\limits_{2r}^{diam\left(
E\right) }\left( 1+\ln \frac{t}{r}\right) t^{n\gamma \left( x\right) -\frac{n%
}{q_{1}\left( x\right) }-1} \\
&&\times \left\Vert f\right\Vert _{L^{p_{1}\left( \cdot \right) }\left( 
\tilde{B}(x,t)\right) }dt \\
&\lesssim &\Vert b\Vert _{\mathcal{C}_{\Pi }^{p_{2}\left( \cdot \right)
,\gamma \left( \cdot \right) }}\left\Vert f\right\Vert _{VL_{\Pi
}^{p_{1}\left( \cdot \right) ,w_{1}\left( \cdot \right) }\left( E\right)
}\sup\limits_{x\in \Pi ,r>0}\frac{1}{w_{2}(x,r)^{\frac{1}{q\left( x\right) }}%
}\int\limits_{r}^{diam\left( E\right) }\left( 1+\ln \frac{t}{r}\right) \\
&&\times t^{\alpha \left( x\right) +n\gamma \left( x\right) }w_{1}^{\frac{1}{%
p_{1}\left( x\right) }}(x,t)\frac{dt}{t} \\
&\lesssim &\Vert b\Vert _{\mathcal{C}_{\Pi }^{p_{2}\left( \cdot \right)
,\gamma \left( \cdot \right) }}\left\Vert f\right\Vert _{VL_{\Pi }^{p\left(
\cdot \right) ,w_{1}\left( \cdot \right) }\left( E\right) }
\end{eqnarray*}%
and 
\begin{equation*}
\lim_{r\rightarrow 0}\sup\limits_{x\in \Pi }\frac{r^{-\frac{n}{q\left(
x\right) }}\left\Vert \left[ b,I_{\Omega ,\alpha \left( \cdot \right) }%
\right] f\right\Vert _{L^{q\left( \cdot \right) }\left( \tilde{B}%
(x,r)\right) }}{w_{2}(x,t)^{\frac{1}{q\left( x\right) }}}\lesssim
\lim_{r\rightarrow 0}\sup\limits_{x\in \Pi }\frac{r^{-\frac{n}{p\left(
x\right) }}\left\Vert f\right\Vert _{L^{p\left( \cdot \right) }\left( \tilde{%
B}(x,r)\right) }}{w_{1}(x,t)^{\frac{1}{p\left( x\right) }}}=0,
\end{equation*}%
which completes the proof of Theorem \ref{teo9}.
\end{proof}

\begin{corollary}
Let $E$, $\Omega $, $p\left( x\right) $, $q\left( x\right) $, $\alpha \left(
x\right) $ be the same as in Theorem \ref{teo7}. Suppose that $q\left( \cdot
\right) $ and $\alpha \left( \cdot \right) $ satisfy (\ref{1}). Then, for $%
\frac{s}{s-1}<p^{-}\leq p\left( \cdot \right) <\frac{n}{\alpha \left( \cdot
\right) }$ and $b\in BMO\left( E\right) $, the following pointwise estimate 
\begin{equation*}
\Vert \left[ b,I_{\Omega ,\alpha \left( \cdot \right) }\right] f\Vert
_{L^{q\left( \cdot \right) }(\tilde{B}(x,r))}\lesssim \Vert b\Vert _{BMO}r^{%
\frac{n}{q\left( x\right) }}\dint\limits_{2r}^{diam\left( E\right) }\left(
1+\ln \frac{t}{r}\right) t^{\frac{n}{q\left( x\right) }-1}\left\Vert
f\right\Vert _{L^{p\left( \cdot \right) }\left( \tilde{B}(x,t)\right) }dt
\end{equation*}%
holds for any ball $\tilde{B}(x,r)$ and for all $f\in L_{loc}^{p\left( \cdot
\right) }\left( E\right) $.

If the functions $w_{1}\left( x,r\right) $ and $w_{2}\left( x,r\right) $
satisfy (\ref{43}) as well as the following Zygmund condition%
\begin{equation*}
\int\limits_{r}^{diam\left( E\right) }\left( 1+\ln \frac{t}{r}\right) \frac{%
w_{1}^{\frac{1}{p\left( x\right) }}(x,t)}{t^{1-\alpha \left( x\right) }}%
dt\lesssim w_{2}^{\frac{1}{q\left( x\right) }}(x,r),\qquad r\in \left(
0,diam\left( E\right) \right]
\end{equation*}%
and additionally these functions satisfy the conditions (\ref{2*})-(\ref{3*}%
), 
\begin{equation*}
d_{\delta }:=\dint\limits_{\delta }^{diam\left( E\right) }\sup_{x\in \Pi
}\left( 1+\ln \frac{t}{r}\right) \frac{w_{1}^{\frac{1}{p\left( x\right) }%
}(x,t)}{t^{1-\alpha \left( x\right) }}dt<\infty ,\qquad \delta >0,
\end{equation*}%
then the operators $\left[ b,I_{\Omega ,\alpha \left( \cdot \right) }\right] 
$ and $\left[ b,M_{\Omega ,\alpha \left( \cdot \right) }\right] $ are $%
\left( VL_{\Pi }^{p\left( \cdot \right) ,w_{1}\left( \cdot \right) }\left(
E\right) \rightarrow VL_{\Pi }^{q\left( \cdot \right) ,w_{2}\left( \cdot
\right) }\left( E\right) \right) $-bounded. Moreover,%
\begin{equation*}
\left\Vert \left[ b,I_{\Omega ,\alpha \left( \cdot \right) }\right]
f\right\Vert _{VL_{\Pi }^{q\left( \cdot \right) ,w_{2}\left( \cdot \right)
}\left( E\right) }\lesssim \Vert b\Vert _{BMO}\left\Vert f\right\Vert
_{VL_{\Pi }^{p_{1}\left( \cdot \right) ,w_{1}\left( \cdot \right) }\left(
E\right) },
\end{equation*}%
\begin{equation*}
\left\Vert \left[ b,M_{\Omega ,\alpha \left( \cdot \right) }\right]
f\right\Vert _{VL_{\Pi }^{q\left( \cdot \right) ,w_{2}\left( \cdot \right)
}\left( E\right) }\lesssim \Vert b\Vert _{BMO}\left\Vert f\right\Vert
_{VL_{\Pi }^{p_{1}\left( \cdot \right) ,w_{1}\left( \cdot \right) }\left(
E\right) }.
\end{equation*}
\end{corollary}

For $\alpha \left( x\right) =0$ in Theorem \ref{teo9}, we get the following
new result:

\begin{corollary}
\label{Corollary2}Let $E$, $\Omega $, $p\left( x\right) $ be the same as in
Theorem \ref{teo7}. Let also $\frac{1}{p\left( \cdot \right) }=\frac{1}{%
p_{1}\left( \cdot \right) }+\frac{1}{p_{2}\left( \cdot \right) }$ and $b\in 
\mathcal{C}_{\Pi }^{p_{2}\left( \cdot \right) ,\gamma \left( \cdot \right)
}\left( E\right) $. Suppose that $p_{1}\left( \cdot \right) $ and $%
p_{2}\left( \cdot \right) $ satisfy (\ref{1}). Then, for $\frac{s}{s-1}%
<p^{-}\leq p\left( \cdot \right) \leq p^{+}<\infty $, the following
pointwise estimate 
\begin{equation*}
\Vert \left[ b,T_{\Omega }\right] f\Vert _{L^{p\left( \cdot \right) }(\tilde{%
B}(x,r))}\lesssim \Vert b\Vert _{\mathcal{C}_{\Pi }^{p_{2}\left( \cdot
\right) ,\gamma \left( \cdot \right) }}r^{\frac{n}{p\left( x\right) }%
}\dint\limits_{2r}^{diam\left( E\right) }\left( 1+\ln \frac{t}{r}\right)
t^{n\gamma \left( x\right) -\frac{n}{p_{1}\left( x\right) }-1}\left\Vert
f\right\Vert _{L^{p_{1}\left( \cdot \right) }\left( \tilde{B}(x,t)\right) }dt
\end{equation*}%
holds for any ball $\tilde{B}(x,r)$ and for all $f\in L_{loc}^{p_{1}\left(
\cdot \right) }\left( E\right) $.

If the function $w\left( x,r\right) $ satisfies (\ref{43}) as well as the
following Zygmund condition%
\begin{equation*}
\int\limits_{r}^{diam\left( E\right) }\left( 1+\ln \frac{t}{r}\right) \frac{%
w^{\frac{1}{p_{1}\left( x\right) }}(x,t)}{t^{1-n\gamma \left( x\right) }}%
dt\lesssim w^{\frac{1}{p\left( x\right) }}(x,r),\qquad r\in \left(
0,diam\left( E\right) \right]
\end{equation*}%
and additionally this function satisfies the conditions (\ref{2*})-(\ref{3*}%
), 
\begin{equation*}
d_{\delta }:=\dint\limits_{\delta }^{diam\left( E\right) }\sup_{x\in \Pi
}\left( 1+\ln \frac{t}{r}\right) \frac{w^{\frac{1}{p_{1}\left( x\right) }%
}(x,t)}{t^{1-n\gamma \left( x\right) }}dt<\infty ,\qquad \delta >0,
\end{equation*}%
then the operators $\left[ b,T_{\Omega }\right] $ and $\left[ b,M_{\Omega }%
\right] $ are $\left( VL_{\Pi }^{p_{1}\left( \cdot \right) ,w\left( \cdot
\right) }\left( E\right) \rightarrow VL_{\Pi }^{p\left( \cdot \right)
,w\left( \cdot \right) }\left( E\right) \right) $-bounded. Moreover,%
\begin{equation*}
\left\Vert \left[ b,T_{\Omega }\right] f\right\Vert _{VL_{\Pi }^{p\left(
\cdot \right) ,w\left( \cdot \right) }\left( E\right) }\lesssim \Vert b\Vert
_{\mathcal{C}_{\Pi }^{p_{2}\left( \cdot \right) ,\gamma \left( \cdot \right)
}}\left\Vert f\right\Vert _{VL_{\Pi }^{p_{1}\left( \cdot \right) ,w\left(
\cdot \right) }\left( E\right) },
\end{equation*}%
\begin{equation*}
\left\Vert \left[ b,M_{\Omega }\right] f\right\Vert _{VL_{\Pi }^{p\left(
\cdot \right) ,w\left( \cdot \right) }\left( E\right) }\lesssim \Vert b\Vert
_{\mathcal{C}_{\Pi }^{p_{2}\left( \cdot \right) ,\gamma \left( \cdot \right)
}}\left\Vert f\right\Vert _{VL_{\Pi }^{p_{1}\left( \cdot \right) ,w\left(
\cdot \right) }\left( E\right) }.
\end{equation*}
\end{corollary}

From Corollary \ref{Corollary2}, we get the following:

\begin{corollary}
Let $E$, $\Omega $, $p\left( x\right) $ be the same as in Theorem \ref{teo7}%
. Then, for $\frac{s}{s-1}<p^{-}\leq p\left( \cdot \right) \leq p^{+}<\infty 
$ and $b\in BMO\left( E\right) $, the following pointwise estimate 
\begin{equation*}
\Vert \left[ b,T_{\Omega }\right] f\Vert _{L^{p\left( \cdot \right) }(\tilde{%
B}(x,r))}\lesssim \Vert b\Vert _{BMO}r^{\frac{n}{p\left( x\right) }%
}\dint\limits_{2r}^{diam\left( E\right) }\left( 1+\ln \frac{t}{r}\right) t^{%
\frac{n}{p\left( x\right) }-1}\left\Vert f\right\Vert _{L^{p\left( \cdot
\right) }\left( \tilde{B}(x,t)\right) }dt
\end{equation*}%
holds for any ball $\tilde{B}(x,r)$ and for all $f\in L_{loc}^{p\left( \cdot
\right) }\left( E\right) $.

If the function $w\left( x,r\right) $ satisfies (\ref{43}) as well as the
following Zygmund condition%
\begin{equation*}
\int\limits_{r}^{diam\left( E\right) }\left( 1+\ln \frac{t}{r}\right) \frac{%
w^{\frac{1}{p\left( x\right) }}(x,t)}{t}dt\lesssim w^{\frac{1}{p\left(
x\right) }}(x,r),\qquad r\in \left( 0,diam\left( E\right) \right]
\end{equation*}%
and additionally this function satisfies the conditions (\ref{2*})-(\ref{3*}%
), 
\begin{equation*}
d_{\delta }:=\dint\limits_{\delta }^{diam\left( E\right) }\sup_{x\in \Pi
}\left( 1+\ln \frac{t}{r}\right) \frac{w^{\frac{1}{p\left( x\right) }}(x,t)}{%
t}dt<\infty ,\qquad \delta >0,
\end{equation*}%
then the operators $\left[ b,T_{\Omega }\right] $ and $\left[ b,M_{\Omega }%
\right] $ are bounded on $VL_{\Pi }^{p\left( \cdot \right) ,w\left( \cdot
\right) }\left( E\right) $. Moreover,%
\begin{equation*}
\left\Vert \left[ b,T_{\Omega }\right] f\right\Vert _{VL_{\Pi }^{p\left(
\cdot \right) ,w\left( \cdot \right) }\left( E\right) }\lesssim \Vert b\Vert
_{BMO}\left\Vert f\right\Vert _{VL_{\Pi }^{p\left( \cdot \right) ,w\left(
\cdot \right) }\left( E\right) },
\end{equation*}%
\begin{equation*}
\left\Vert \left[ b,M_{\Omega }\right] f\right\Vert _{VL_{\Pi }^{p\left(
\cdot \right) ,w\left( \cdot \right) }\left( E\right) }\lesssim \Vert b\Vert
_{BMO}\left\Vert f\right\Vert _{VL_{\Pi }^{p\left( \cdot \right) ,w\left(
\cdot \right) }\left( E\right) }.
\end{equation*}
\end{corollary}

\section{Acknowledgement}

F. G\"{u}rb\"{u}z was partly supported by the grant of Hakkari University
Scientific Research Project (No. FM18BAP1) under the research project "Some
estimates for rough Riesz type potential operator with variable order and
rough fractional maximal{\ operator with variable order both on }generalized
variable exponent Morrey spaces and vanishing generalized variable exponent
Morrey spaces". The researches of the other authors were partly supported by
Institution of Higher Education Scientific Research Project in Ningxia (No.
NGY2017011) and Natural Science Foundation of China (No. 11461053;
11762017), respectively.

\end{document}